\documentclass[a4paper,11pt,reqno]{amsart}
\usepackage[utf8]{inputenc}
\usepackage[french, english]{babel}
\usepackage{mathtools}
\usepackage{amssymb}
\usepackage{amsfonts}
\usepackage{amsthm}
\usepackage{bbm}
\usepackage{fullpage}
\usepackage{color}
\usepackage{hyperref}
\usepackage{stmaryrd}
\usepackage{enumitem}
\usepackage{accents}
 \mathtoolsset{showonlyrefs}
\newcommand{\N}{\mathbb N}
\newcommand{\Z}{\mathbb Z}

\newcommand{\eps}{\varepsilon}
\renewcommand{\P}{\mathbb P}

\newcommand{\F}{\mathcal F}

\newcommand{\ind}{\mathbbm 1}
\usepackage{pstricks,pst-node,pst-text,pst-3d}
\usepackage{graphics,epsf,psfrag}
\newtheorem{theorem}{Theorem}[section]
\newtheorem{lemma}[theorem]{Lemma}
\newtheorem{proposition}[theorem]{Proposition}
\newtheorem{corollary}[theorem]{Corollary}

\newtheorem{thmx}{Theorem}

\newtheorem{remark}[theorem]{Remark}

\theoremstyle{definition}

\newcommand{\dd}{\text{d}}

\newcommand{\f}{\mathfrak{f}}

\newcommand{\esssup}{\operatorname{ess\,sup}}

\renewcommand{\tilde}{\widetilde}

\newcommand{\bbE}{{\ensuremath{\mathbb E}} }

\newcommand{\bbN}{{\ensuremath{\mathbb N}} }

\newcommand{\bbP}{{\ensuremath{\mathbb P}} }

\newcommand{\bbR}{{\ensuremath{\mathbb R}} }

\newcommand{\bbZ}{{\ensuremath{\mathbb Z}} }

\renewcommand{\epsilon}{\varepsilon}

\newcommand{\gb}{\beta}
\newcommand{\gep}{\varepsilon}       % \ge already exists...

\newcommand{\go}{\omega}

\newcommand{\gl}{\lambda}

\newcommand{\cF}{{\ensuremath{\mathcal F}} }

\newcommand{\cT}{{\ensuremath{\mathcal T}} }

\newcommand{\bP}{{\ensuremath{\mathbf P}} }

\newcommand{\bE}{{\ensuremath{\mathbf E}} }
\newcommand{\bH}{{\ensuremath{\mathbf H}} }

\newcommand{\lint}{\llbracket}
\newcommand{\rint}{\rrbracket}

\newcommand{\maxtwo}[2]{\max_{\substack{#1 \\ #2}}} % max with 2 lines
\newcommand{\sumtwo}[2]{\sum_{\substack{#1 \\ #2}}} % sum with 2 lines
\definecolor{darkred}{rgb}{0.7,0.1,0.1}
\renewcommand{\hat}{\widehat}

\renewcommand{\complement}{\mathsf{c}}

\begin{document}

\author{Stefan Junk}
\address{Gakushuin University, 1-5-1 Mejiro, Toshima-ku, Tokyo 171-8588 Japan}
\email{sjunk@math.gakushuin.ac.jp}
\author{Hubert Lacoin}
\address{IMPA, Estrada Dona Castorina 110, 
Rio de Janeiro RJ-22460-320- Brasil}
\email{lacoin@impa.br}

\title{Strong disorder and very strong disorder are equivalent for directed polymers}

\begin{abstract}
We show that if the normalized partition function $W^{\beta}_n$ of the directed polymer model on $\bbZ^d$ converges to zero, then it does so exponentially fast. This implies that there exists a critical value $\beta_c$ for the inverse temperature such that the normalized partition function has a non-degenerate limit for all $\beta\in [0,\beta_c]$ -- weak disorder holds -- while for $\beta\in (\beta_c,\infty)$ it converges exponentially fast to zero -- very strong disorder holds. This solves a twenty-years-old conjecture formulated by Comets, Yoshida, Carmona and Hu.
Our proof requires a technical assumption on the environment, namely, that it is bounded from above.
   \\[10pt]
\end{abstract}

 \maketitle

\section{Introduction}

The \textit{directed polymer in a random environment} on $\bbZ^d$ is a statistical mechanics model, which has  been introduced, for $d=1$, in \cite{HH85} as a toy model to study the interfaces of the planar Ising model with random coupling constants. The model was shortly afterwards generalized to higher dimension 
(see for instance \cite{B89,IS88}). When $d\ge 2$, rather than an effective interface model, the directed polymer in a random environment can be interpreted as modeling the behavior of a stretched polymer in a solution with impurities.
The interest in the model, triggered by its rich phenomenology (including connections with the Stochastic heat equation and the Kardar-Parisi-Zhang (KPZ) equation, see \cite[Section 8]{C17}) has since then generated a plentiful literature in theoretical physics and mathematics.  We refer to \cite{C17,Zy24} for recent reviews on the subject.

\medskip

An important topic for the directed polymer is the so-called localization transition, which can be described as follows. At low temperature, the polymer trajectories are localized in a narrow corridor where the environment is most favorable \cite{B18,Bates21,BC20,CH02,CH06,CSY03}. At high temperature, the effect of the disorder disappears on large scale, and the rescaled polymer trajectories have the same scaling limit as the simple random walk, that is to say, standard Brownian motion \cite{B89,CY06,IS88,SZ96}.

\medskip

This transition can be defined in terms of the asymptotic behavior of the renormalized partition function of the model $W^{\beta}_n$ (see \eqref{defwn} in the next section). If $W^{\beta}_n$ converges to an almost surely positive limit $W^\beta_{\infty}$ we say that \textit{weak disorder} holds. On the other hand if $W^{\beta}_n$ converges to zero almost surely, we say that \textit{strong disorder} holds. It has been proved in \cite{CY06} that weak disorder implies that the distribution of the polymer  converges  after diffusive rescaling to that of standard Brownian motion (see also \cite{H24_1} for an alternative short and self-contained proof). While some localization results have been proved under the strong disorder assumptions, for instance in \cite{CH06}, much stronger characterizations of disorder-induced localization have been obtained under the stronger assumption that $W^{\beta}_n$ converges to zero
exponentially fast ($\lim_{n\to \infty} n^{-1} \log W^{\beta}_n<0$), see for instance \cite{B18,BC20}. This latter regime is known as the \textit{very strong disorder} regime.

\medskip

When $d=1$ or $2$ it has been proved that very strong disorder holds for every $\beta>0$ \cite{CV06,L10} (it was proved earlier in \cite{CH02,CSY03} that strong disorder holds for every $\beta>0$). On the other hand, when $d\ge 3$, it has been shown \cite{CY06} that there exist $\beta_c,\bar \beta_c \in(0,\infty]$ such that:
\begin{itemize}
 \item weak disorder holds for $\beta\in [0,\beta_c)$;
 \item strong disorder holds for $\beta>\beta_c$;
 \item very strong disorder holds if and only if $\beta>\bar \beta_c$.
\end{itemize}

It has been conjectured \cite{BC20,CH06,C17,CY06,L10,V23,Zy24} that $\beta_c=\bar\beta_c$. This conjecture implies that, at least in a mild sense, the two notions of strong disorder are equivalent. 
The present paper proves this conjecture and also establishes that weak disorder holds at $\beta_c$.

\medskip

Our result is an example of a \textit{sharp phase transition}, since $W^{\beta}_n$ either converges to some positive (random) limit -- if $\beta\in[0,\beta_c]$ -- or  decays exponentially fast -- if $\beta\in(\beta_c,\infty)$. This excludes the existence of an intermediate phase. This behavior is similar to that of the long-range connection probability for Bernoulli percolation \cite{AB87,Men86} and to that of spin-spin correlation for the Ising model \cite{ABF87}. We refer to \cite{DCT16} for a more recent shorter proof and to  \cite{DRT19} and to \cite{DGRST23} for recently obtained sharpness results concerning Potts model and Random interlacements, respectively.

  \medskip
 
Combined with the recent observations made on the tail distribution of $W^{\beta}_{\infty}$ in the weak disorder phase  \cite{J22,J23, J21_2, JL24_2}, our result provides a new perspective on the phase transition from weak to strong disorder for directed polymers when $d\ge 3$. The latter is, for the time being, is largely unchartered territory. Indeed, it our results identifies exactly the tail distribution of the limiting partition function at criticality $W^{\beta_c}_{\infty}$ (see Corollary~\ref{unplusd} and  \cite[Theorem 2.1]{JL24_2}). This information might allow an analysis of the behavior of the directed polymer at $\beta_c$ by identifying features that are specific to the model at criticality. Furthermore, it should help in controlling the rate of decay of $W^{\beta}_n$ for $\beta$ just slightly above $\beta_c$, hence establishing regularity estimates for the free energy at the vicinity of the critical point similar to those known to hold in dimensions $d=1$ and $d=2$ (see \cite{BL15,Naka19} for the sharpest established results).
 
\medskip

Our proof relies on two main technical innovations.  The first is to reduce the problem to a question about end-point localization. More precisely, we reduce the proof of the equivalence of strong and very strong disorder to  the following claim: when $W^{\beta}_n$ reaches its supremum, then, most likely, the endpoint polymer measure is localized (see Theorem~\ref{locaend}). This implication is obtained by chaining a variety of arguments. Some of these (e.g.\ fractional moments, spine constructions) have already been used in the study of directed polymer and disordered system, while others are novel. The reasoning is explained in Sections~\ref{prelim} to ~\ref{proofalsmostthesame}.
 The second innovation is the proof of this localization statement, based on discrete-time stochastic calculus. It is detailed in Sections~\ref{intermezzo}--\ref{discretos}.

\section{Model and results}

\subsection{Weak, strong and very strong disorder for directed polymer}
We let $X=(X_k)_{k\ge 0}$ be the nearest neighbor simple random walk on $\bbZ^d$ starting from the origin, and $P$ the associated probability distribution. It is the Markov chain satisfying $P(X_0=0)=1$  and
\begin{equation}\label{defDD}
P(X_{n+1}=y \ | \ X_n=y)=\frac{1}{2d}\ind_{\{|y-x|=1\}}=:D(x,y),
\end{equation}
where $|\cdot|$ denotes the $\ell_1$ distance on $\bbZ^d$.
Given a collection $\go=(\go_{k,x})_{k\ge 1, x\in \bbZ^d}$ of real-valued weights (the environment), $\beta\ge 0$ (the inverse temperature) and  $n\ge 1$ (the polymer length), we define the polymer measure $P^\beta_{\go,n}$
as a modification of the distribution $P$ which favors trajectories that visit sites where $\go$ is large.
More precisely,
to each path $\pi\colon \N\to\Z^d$ we associate an energy
	$H_n(\omega,\pi)\coloneqq \sum_{i=1}^n \omega_{i,\pi(i)},$
and we set
\begin{align}\label{eq:mu}
	P_{\omega,n}^\beta(\dd X)\coloneqq\frac{1}{Z_{\go,n}^\beta}e^{\beta H_n(\omega,X)}P(\dd X)  \quad  \text{ where } \quad Z_{\go,n}^\beta\coloneqq  E\left[  e^{\beta H_n(\omega,X)}\right].
\end{align}
The quantity $Z_{\go,n}^\beta$ is referred to as the partition function of the model.
In what follows, we assume that the environment $\go$ is given by a fixed realization of an i.i.d.\ random field on $\bbZ^d$ and we let $\bbP$ denote the associated probability distribution. We assume that $\go$ has finite exponential moments of all orders, that is
\begin{equation}\label{expomoment}
 \forall \beta\in \bbR, \ \gl(\beta)\coloneqq \log \bbE[ e^{\beta\go_{1,0}}]<\infty.
\end{equation}
The reader can readily check that  in this setup we have
$\bbE[Z_{\go,n}^\beta]= e^{n\gl(\beta)}.$
We thus define the normalized partition function $W_{n}^\beta$ (the dependence in $\go$ is dropped from the notation for convenience) by setting
\begin{equation}\label{defwn}
 W_{n}^\beta= \frac{Z_{\go,n}^\beta}{\bbE[Z_{\go,n}^\beta]}
 =E\left[  e^{\beta H_n(\omega,X)-n\gl(\beta)}\right].
\end{equation}
The normalized partition function  $W_{n}^\beta$ encodes essential information on the typical behavior of $X$ under $P^{\go}_{\go,n}$ and has thus been a central object of attention in polymer studies.
A crucial  observation made in  \cite{B89} is that $W^{\beta}_n$ is a non-negative martingale with respect to the natural filtration  $(\cF_n)$ associated with $\go$ ($\F_n\coloneqq \sigma((\omega_{k,x})_{k\leq n, x\in \bbZ^d})$) and hence converges almost surely.
Furthermore setting $W^{\beta}_{\infty}\coloneqq \lim_{n\to\infty} W^{\beta}_n$, it has been noticed  that
 $\P(W_\infty^\beta>0)\in\{0,1\}$, since the event belongs to the tail $\sigma$-algebra.

 \medskip

 When $W^\beta_{\infty}>0$ almost surely, we say that \textit{weak disorder} holds while if $\bbP( W^\beta_{\infty}=0)=1$, we say that \textit{strong disorder} holds.
A third notion related to the influence of disorder is associated with the free energy, which is defined as the asymptotic rate of decay of $W^{\beta}_n$,
\begin{equation}\label{freen}
	\f(\beta)\coloneqq \lim_{n\to\infty}\frac 1n \bbE\left[\log W_n^\beta\right]= \lim_{n\to\infty}\frac 1n\log W_n^\beta= \sup_{n\ge 1} \frac 1n \bbE\left[\log W_n^\beta\right].
\end{equation}
In the above, the second convergence holds $\bbP$-almost surely and the existence of the limits was established in \cite[Proposition 2.5]{CSY03}. Using  H\"older's inequality, it is easy to check that $\beta\mapsto \f(\beta)+\lambda(\beta)$ is convex and hence continuous. As a consequence $\beta\mapsto\f(\beta)$ is continuous.
Following a terminology introduced in \cite{CH06} we say that \textit{very strong disorder} holds if $\f(\beta)<0$.

\smallskip It has been shown in \cite{CY06} that the influence of the disorder \textit{increases} with $\beta$ in the following sense (in this statement $\beta_c$ and $\bar \beta_c$ depend of course on $d$ and on the distribution of $\go$).

 \begin{thmx}[{\cite[Theorem 3.2]{CY06}}]\label{thmx:CY}
 The following hold:
  \begin{itemize}
   \item [(a)]  There exists  a $\beta_c\in [0,\infty]$ such that weak disorder holds for $\beta< \beta_c$
   and strong disorder holds for $\beta>\beta_c$.
\item [(b)] The function $\beta \mapsto \f(\beta)$ is non-increasing. In particular, by continuity there exists a $\bar \beta_c\in[0,\infty]$ such that very strong disorder holds if and only if $\beta>\bar \beta_c$.
  \end{itemize}

 \end{thmx}
\begin{remark} While it is possible to have $\beta_c=\infty$, it rather the exception than the rule. It has been shown in \cite[Theorem 2.3]{CSY03} that $\beta_c,\bar \beta_c<\infty$ whenever $\bbP[\go_{1,0}=\esssup \go_{1,0}]<\frac{1}{2d}$. 
      \end{remark}

\subsection{The conjecture}

 Very strong disorder implies that the normalized partition function converges to zero exponentially fast. Hence \textit{very strong disorder implies strong disorder} and in particular $\bar \beta_c \ge \beta_c$. 
 It has long been conjectured that we have $\bar \beta_c = \beta_c$ and thus that there is only one phase transition from
 weak disorder to very strong disorder. As far as directed polymers on $\bbZ^d$ are concerned, this has only been proved to be true for $d=1$ \cite{CV06}
 and $d=2$ \cite{L10} by showing that in both cases $\bar \beta_c=\beta_c=0$.
 
 \medskip
 
 The conjecture appeared first in \cite{CH06,CY06}, and has since been regularly mentioned as a challenging problem  in the polymer literature, see for instance  \cite{BC20,  C17, L10, LM10,  V23,Zy24}. In many of these references, the conjecture does not include a prediction on whether strong or weak disorder  holds at $\beta_c$. It is worth noting that there has been no consensus on the subject: for example, \cite{CH06} predicts weak disorder at $\beta_c$, while \cite[Conjecture 3.1]{C17} predicts strong disorder. In \cite{Zy24}, the matter was presented as an open question before the announcement of our result, with the comment ``\textit{Currently no concrete guess or heuristics exist about the answer to this question}''.

 \subsection{Main result}
 
We solve the above conjecture and further establish that weak disorder holds at $\beta_c$ when $d\ge 3$. Both results are obtained under the technical assumption that the environment is bounded from above, that is, that there exists $A\in \bbR$ such that
$\bbP[\go_{1,0}\le A]=1.$ This assumption is used in several steps of the proof, but the result and most of the argument carry through as long as \eqref{expomoment} is satisfied (see Section~\ref{new_lit}).

\begin{theorem}\label{main}
For any $d\ge 3$, if the environment is bounded from above, then strong disorder and very strong disorder are equivalent. In other words, the following hold:
\begin{itemize}
 \item [(a)] Equality of the critical points, that is $\beta_c=\bar \beta_c$.
 \item [(b)] If $\beta_c<\infty$, then weak disorder holds at criticality, that is $W^{\beta_c}_{\infty}>0$ $\bbP$-almost surely.
\end{itemize}
\end{theorem}

Let us compare our theorem with related results in the literature. It was shown in \cite{BP93} that $\beta_c=\bar \beta_c$ for directed polymer on trees (introduced earlier in \cite{kahane74, Man74} as multiplicative cascades). More precisely, for this model the explicit value of $\beta_c$ can be computed, see \cite[Theorem 1]{kahane74}, and the same result implies that \textit{strong disorder} holds at $\beta_c$, in contrast with item $(b)$ in our theorem.

\smallskip In \cite{LM10}, the equivalence of strong and very strong disorder (including the fact that weak disorder holds at $\beta_c$) is established for a hierarchical version of the model defined on recursively constructed ``diamond'' lattices. Hierarchical versions of models in statistical mechanics have been used by physicists as a non-mean-field simplification of models defined on $\bbZ^d$ (see for instance \cite{CD89,DIL84}).
The big advantage of a hierarchical model is that the partition function satisfies a simple recursion. The equivalence of strong and very strong disorder in this setup is a rather direct consequence of this recursion (see \cite[Proposition 5.1]{LM10}).

\subsection{A consequence for the partition function at $\beta_c$}\label{consequencesec}
Let us discuss shortly here a consequence of our result.
Let us define $p^*(\beta)$ as 
\begin{equation}
 p^*(\beta):=\sup\left\{p>0 \colon \sup_{n\ge 1} \bbE\left[(W^{\beta}_n)^p\right]<\infty \right\}.
\end{equation}
It is easy to check that we have $p^*(\beta)=1$ in the strong disorder phase. 
In the weak disorder phase we have the following alternative characterization
$$  p^*(\beta)=\sup\left\{p>0 \colon  \bbE\left[(W^{\beta}_\infty)^p\right]<\infty \right\}.$$
It has been proved (see \cite[Corollary 1.3]{J22} and \cite[Corollary~2.8(ii)]{JL24_2}) that we have $p^*(\beta)\ge 1+\frac{2}{d}$ in the weak disorder phase.
It has further been shown  that $\beta\mapsto p^*(\beta)$ is right-continuous  at points for which $p^*(\beta)\in \left(1+\frac{2}{d},2\right]$
(see \cite[Theorem 1.2]{J23}) and left-continuous at points which are such that $p^*(\beta)>1$ (see \cite[Corollary~2.11(i)]{JL24_2}).
%  Here, $\beta_2$ is the $L^2$-critical point defined in Section~\ref{l2consi} and we have $\beta_c>\beta_2$ by Theorem~\ref{b2bcrit}.
The combination of these observations allows to identify the value of $p^*(\beta)$ at the critical point.
\begin{corollary}\label{unplusd}
%  If the environment is bounded from above, then
For every $d\ge 3$, we have 
\begin{equation} \lim_{\beta\uparrow \beta_c}p^*(\beta)= p^*(\beta_c)=1+\frac{2}{d}.\end{equation} 
 
\end{corollary}
\begin{proof}
 From a simple $L^2$-computation (see Section~\ref{l2consi}) we can check that $p^*(\beta_c)\le 2$. On the other hand, since weak disorder holds at $\beta_c$ we have  $p^*(\beta_c)\ge 1+\frac{2}{d}$ by \cite[Corollary~2.8(ii)]{JL24_2}. In particular, since  $p^*(\beta)=1$ for $\beta>\beta_c$, $\beta\mapsto p^*(\beta)$ is discontinuous to the right of $\beta_c$. From \cite[Theorem 1.2]{J23}, this discontinuity implies that $p^*(\beta_c)\notin (1+\frac{2}{d},2]$. Thus the only possible value is $p^*(\beta_c)=1+\frac{2}{d}$. Continuity from the left at $\beta_c$ is then also a consequence of  \cite[Corollary~2.11(i)]{JL24_2}.
\end{proof}

\subsection{Literature update}\label{new_lit}

As a conclusion to this introduction we mention related results that have been obtained after the first version of this manuscript appeared online. 

\subsubsection{Extensions of  Theorem~\ref{main}}
 In \cite{JL25}, the validity of Theorem~\ref{main} has been extended in two directions. Firstly, we have removed the assumption that $\go$ is bounded from above, extending the result to all environment satisfying \eqref{expomoment}. We also consider more general $\bbZ^d$-valued random walks $(X_n)_{n\ge 1}$ with i.i.d.\ increments and such that $E[|X_1|^{\eta}]<\infty$ for some $\eta>0$. In this more general framework, we establish that $\beta_c=\bar \beta_c$, but could not prove in full generality that weak disorder holds at $\beta_c$. This mostly because we lack an adequate extension of Theorem~\ref{b2bcrit} (we refer to \cite[Proposition~2.3 and Remark~2.4]{JL25} for more details).

 \medskip
 
One of the key steps to prove  Theorem~\ref{main} for general environment is to ensure  that Lemma~\ref{dull} remains valid for unbounded environments. This was achieved in a separate work, see \cite[Proposition 2.3]{JL24_2}.

\subsubsection{Regularity of the phase transition}

The combination of Corollary~\ref{unplusd} with  \cite[Theorem~2.1]{JL24_2} (which appeared a couple of months after this manuscript) provides very precise information concerning the distribution of $W^{\beta_c}_{\infty}$. More precisely, we have  $\bbP( W^{\beta_c}_{\infty}>u) \asymp u^{-1-\frac{2}{d}}$ as $u\to \infty$. In \cite{H25}, this information is used as starting point to investigate the regularity of the free-energy $\f(\beta)$ in the vicinity of the critical point. It is proved there that the transition is smooth in the sense that
\begin{equation}\label{smoooth}
\lim_{\beta \downarrow \beta_c} \frac{\log(|\f(\beta-\beta_c)|)}{\log (\beta-\beta_c)}=\infty.
\end{equation}
In words, $\f(\beta)$ is smaller than any power of $(\beta-\beta_c)$ in the neighborhood of $\beta_c$. It remains a challenge to obtain more quantitative version of \eqref{smoooth}.

\subsection{Organization of the paper}

Before giving a complete overview about the organization of the proof, we need to introduce a couple of technical elements. In Section~\ref{prelim}, we explain how the result can be deduced by bounding the probability of a well-chosen event under both the original law of the environment $\P$ and under the \textit{size-biased} measure. In Section~\ref{orgamain}, we introduce the event whose probability needs to be controlled and provide a backbone for the proof, introducing all the relevant technical estimates. At the end of Section~\ref{orgamain}, we explain how the remainder of the paper is organized.

\section{Preliminaries}\label{prelim}

\subsection{Notation}
 The point-to-point partition function $\hat W^{\beta}_n(x)$ is defined as
\begin{equation}
 \hat W^{\beta}_n(x)\coloneqq   E\left[  e^{\beta H_n(\omega,X)}\ind_{\{X_n=x\}}\right].
\end{equation}
The endpoint distribution of the polymer $\mu^{\beta}_{n}$ is defined as a  probability on $\bbZ^d$ as follows
\begin{equation}\label{endpoint}
\mu^{\beta}_{n}(x)\coloneqq P^{\beta}_{\go,n}(X_n=x)= \frac{ \hat W^{\beta}_n(x)}{ W^{\beta}_n}.
\end{equation}
The distribution of $X_n$ under the measure $P^{\beta}_{n-1}$ also appears naturally in computations. It can be expressed as a function of $\mu^{\beta}_{n-1}$. We have (recall \eqref{defDD} and observe that $D$ is symmetric)
\begin{equation}
P^{\beta}_{\go,n-1}(X_n=x)= \sum_{y\in \bbZ^d} \mu^{\beta}_{n-1}(y) D(y,x)= (D  \mu^{\beta}_{n-1}) (x).
\end{equation}
We define the end-point replica overlap  by
\begin{equation}\label{overlap}
 I_n\coloneqq  \sum_{x\in \bbZ^d}D  \mu^{\beta}_{n-1} (x)^2.
\end{equation}
The quantity $I_n$ appears when computing the conditional variance of the partition function,
\begin{equation}\begin{split}\label{condvar}
 \bbE\left[ (W^{\beta}_{n}-W^{\beta}_{n-1})^2 \ \middle| \ \cF_n \right]&= (W^{\beta}_{n-1})^2 \bbE\left[ \left(\frac{W_n}{W_{n-1}}\right)^2-1 \ \middle| \ \cF_n \right] \\
 &=(W^{\beta}_{n-1})^2  \bbE\left[ (E^{\beta}_{\go,n-1})^{\otimes 2} 
 \left[ e^{\beta\go(n,X^{(1)}_n)+ \beta\go(n,X^{(2)}_n)-2\gl(\beta)}-1\right]\right]\\
 &= \chi(\beta) (W^{\beta}_{n-1})^2 I_n,
 \end{split}\end{equation}
 where $\chi(\beta)\coloneqq e^{\gl(2\beta)-2\gl(\beta)}-1$ is the variance of $e^{\beta \go_{1,0}-\gl(\beta)}$.
We set
 $L\coloneqq \esssup( e^{\beta \go_{1,0}-\gl(\beta)})$ (recall that $L<\infty$ by assumption).
Note that for every $n\ge 0$ we have, almost surely,
\begin{equation}\label{fromabove}
 W^{\beta}_{n+1}\le LW^{\beta}_{n}.
\end{equation}
Finally, given $m\ge 1$ and $y\in \bbZ^d$, we define $\theta_{m,y}$ to be the shift operator acting on $\go$. The environment  $\theta_{m,y}\  \go$ is a field indexed by $\bbN\times \bbZ^d$ defined by
\begin{equation}\label{shift1}
 (\theta_{m,y}\  \go)_{n,x}\coloneqq  \go_{n+m, x+y}.
\end{equation}
We also let $\theta_{m,y}$ act on random variables and events that are defined in terms of $\go$ by setting
\begin{equation}\label{shift2}
 \theta_{m,y}\  f(\go)\coloneqq  f(\theta_{m,y}\  \go) \quad  \text{ and } \quad \theta_{m,y} A\coloneqq  \{ \go \ \colon \ \theta_{m,y}\  \go\in A\}.
\end{equation}

\subsection{The $L^2$-regime}\label{l2consi}
In \cite{B89}, the existence of a weak disorder phase is obtained via a  second moment  computation which we briefly sketch. 
Using Fubini's theorem and  monotone convergence,  we have

\begin{equation}
\lim_{n\to \infty}\bbE\left[ (W^{\beta}_n)^2\right]= E^{\otimes 2} \left[ e^{(\gl(2\beta)-2\gl(\beta))L(X^{(1)},X^{(2)})}\right],
\end{equation}
where $X^{(1)}$ and $X^{(2)}$ are two independent copies of the random walk $X$ and $L$ denotes the intersection time between the two trajectories
$ L(X^{(1)},X^{(2)})=\sum_{k=1}^{\infty} \ind{\{X_k^{(1)}=X_k^{(2)}\}}.$
By the Markov property, $L(X^{(1)},X^{(2)})$ is a geometric random variable. Hence the above limit is finite if and only if
\begin{equation}
 e^{\gl(2\beta)-2\gl(\beta)}< \frac{1}{P^{\otimes 2}[\exists n\ge 1, X^{(1)}_n=X^{(2)}_n] }=  \frac{1}{E^{\otimes 2}\left[L_\infty(X^{(1)},X^{(2)}) \right]}+1.
\end{equation}
Hence $(W^\beta_n)$ is bounded in $L^2$ if and only if $\beta< \beta_2$, where (recall that $\chi(\beta)\coloneqq e^{\gl(2\beta)-2\gl(\beta)}-1)$
\begin{equation}\label{defbeta2}
\beta_2\coloneqq \sup\left\{\beta \ \colon \chi(\beta) E^{\otimes 2}\big[L(X^{(1)},X^{(2)}\}\big]< 1\right\}\in(0,\infty].
\end{equation}An important observation is that since $L^2$-boundedness implies that  $W^{\beta}_n$ converges in $L^2$, we have $\beta_c \ge \beta_2$. In particular we have $\beta_c>0$ for $d\ge 3$ and, using the terminology introduced in Section~\ref{consequencesec},  $p^*(\beta_c)\le p^*(\beta_2)\le 2$.
In order to prove part $(b)$ of our main result -- weak disorder at criticality -- we need to ensure that weak disorder holds at $\beta_2$, which is a consequence of  the following stronger result.

\begin{thmx}\label{b2bcrit}
In dimension $d\ge 3$, we have $\beta_c>\beta_2$. In particular, if strong disorder holds at $\beta$, then $\beta>\beta_2$.
\end{thmx}

The proof of Theorem~\ref{b2bcrit} is presented in \cite[Section 1.4]{BS10} in the case when $d\ge 4$. When $d=3$, as observed in \cite[Remark 5.2]{C17} and \cite[Section 3.2]{Zy24}, Theorem~\ref{b2bcrit} can easily be  obtained by combining the content of \cite{BT10} with the connection made in \cite{B04} between directed polymer and the random walk pinning model (an alternative approach, also mentioned in \cite{BS10,C17} is to exploit a similar connection between directed polymer and disordered renewal pinning model and use results from \cite{GLT11}).
For the convenience of the reader we provide in Appendix~\ref{gapfilling} a guideline on how to deduce Theorem~\ref{b2bcrit} from the content of \cite{BT10}.

\subsection{A finite volume criterion relying on fractional moments}

By Jensen's inequality we have (we omit the superscript $\beta$ when writing square roots to lighten the notation)
\begin{equation}\label{tops}
 \bbE[ \log W_n]=  \bbE\left[ 2\log \sqrt{W_n}\right]\le 2 \log \bbE\left[ \sqrt{W_n}\right].
\end{equation}
Hence, to show that very strong disorder holds, it is sufficient to show that $\bbE\left[ \sqrt{W_n}\right]$ decays exponentially fast in $n$.
From \cite{CV06}, we see that exponential decay holds as soon as $\bbE\left[ \sqrt{W_n}\right]$ is sufficiently small for  some $n$, or more precisely:

\begin{proposition}\label{finitevol}
 If there exists  an $n$  such that
\begin{equation}
 \bbE\left[ \sqrt{W_n}\right]\leq \frac 1{\sqrt 2(2n+1)^{d}},
\end{equation}
then very strong disorder holds,  or more precisely $\f(\beta)\leq -\frac{\log 2}{n}$.
\end{proposition}
The result is an immediate consequence of \cite[Theorem 3.3]{CV06}. We include a proof for the convenience of the reader. \begin{proof}
From \eqref{tops}, we have
\begin{equation}\label{keyz}
 \f(\beta)\le \limsup_{m\to \infty} \frac{2}{nm} \log \bbE\left[ \sqrt{W_{nm}}\right].
\end{equation}
  Given $x_1,\dots,x_m\in \bbZ^d$, we let $\hat W_{nm}(x_1,x_2,\dots,x_m)$ denote the contribution to the partition function of
 trajectories that go through $x_1,x_2,\dots, x_m$ at times $n,2n,\dots, nm$, that is
$$\hat W_{nm}(x_1,x_2,\dots,x_m)\coloneqq  \bbE\left[ e^{H_{nm}(\go,X)}\ind_{\{\forall i\in \lint 1,m\rint, S_{ni}=x_i\}}\right]$$
Using the inequality $\sqrt{ \sum_{i\in I} a_i}\le \sum_{i\in I} \sqrt{a_i}$, which is valid for an arbitrary collection of non-negative numbers $(a_i)_{i\in I}$, we obtain
\begin{equation}\begin{split}\label{factor}
 \bbE\left[ \sqrt{W_{nm}}\right]&\le \sum_{(x_1,\dots,x_m)\in (\bbZ^d)^m} \bbE\left[\sqrt{\hat W_{nm}(x_1,x_2,\dots,x_m)} \right]\\
 &= \sum_{(x_1,\dots,x_m)\in (\bbZ^d)^m} \prod_{i=1}^m
 \bbE\left[\sqrt{\hat W_{n}(x_{i}-x_{i-1})} \right]\\
 &=\left(\sum_{x\in \bbZ^d}\bbE\left[ \sqrt{\hat W_n(x)}\right]\right)^m,
\end{split}\end{equation}
where the factorization in the  second line is obtained by combining the Markov property for the random walk with the independence of the environment. The equality in the last line follows from a rearrangement of the terms.
Now, since $\hat W_n(x)\le W_n$, we have
\begin{equation}\label{discase}
 \sum_{x\in \bbZ^d} \bbE\left[ \sqrt{\hat W_n(x)}\right]= \sum_{|x|\le n} \bbE\left[ \sqrt{\hat W_n(x)}\right] \le (2n+1)^d \bbE\left[ \sqrt{W_n}\right]
\end{equation}
and hence combining \eqref{keyz}-\eqref{factor} and \eqref{discase} we obtain that
\begin{equation}
  \f(\beta)\le \frac{2}{n} \log\left( (2n+1)^d \bbE\left[ \sqrt{W_n}\right]\right),
\end{equation}
and hence the result.\end{proof}

\subsection{Fractional moment and size-biasing}\label{sizeb}

We consider the size-biased measure for $\go$ defined by
$ \tilde \bbP_n(\dd \go)\coloneqq  W^{\beta}_n\bbP(\dd \go).$
The following provides a simple manner to bound   $\bbE\left[ \sqrt{W_n}\right]$.
\begin{lemma}\label{babac}

 For any measurable event $A$, it holds that
 \begin{equation}
  \bbE\left[ \sqrt{W_n}\right] \le \sqrt{\bbP(A)} + \sqrt{\tilde \bbP_n(A^{\complement})},
\end{equation}

\end{lemma}

Note that in order to prove that $\f(\beta)<0$, by combining Proposition~\ref{finitevol} with Lemma~\ref{babac}, it is sufficient to find an event $A$ which is unlikely under the original measure $\bbP$ and typical under the size-biased measure $\tilde \bbP_n$.

\begin{proof}
For any measurable positive function $f$, we have, using the Cauchy-Schwarz inequality,
\begin{equation}
    \bbE\left[ \sqrt{W_n}\right]^2\le     \bbE\left[ W_nf(\go)\right] \bbE\left[ f(\go)^{-1}\right]= \tilde\bbE_n\left[ f(\go)\right] \bbE\left[ f(\go)^{-1}\right].
\end{equation}
We consider $f$ given by
 $f\coloneqq  \alpha \ind_A + \alpha^{-1} \ind_{A^{\complement}}$
with $\alpha= (\bbP(A) \tilde \bbP_n(A^{\complement}))^{1/4}$.
We have
\begin{equation}\begin{split}
 \tilde\bbE_n\left[ f(\go)\right] \bbE\left[ f(\go)^{-1}\right]&
 =\left( \alpha \tilde \bbP_n(A)+\alpha^{-1}\tilde \bbP_n(A^{\complement})\right)
 \left( \alpha^{-1} \bbP(A)+\alpha \bbP(A^{\complement})\right)\\
 &\le \left( \alpha +\alpha^{-1}\tilde \bbP_n(A^{\complement})\right)
 \left( \alpha^{-1} \bbP(A)+\alpha \right)\\
 &=2\sqrt{\bbP(A) \tilde \bbP_n(A^{\complement})}+  \bbP(A) +  \tilde \bbP_n(A^{\complement})\\
 &= \left( \sqrt{\bbP(A)} + \sqrt{\tilde \bbP_n(A^{\complement})}\right)^2.
\qedhere \end{split}\end{equation}
\end{proof}

\subsection{Spine representation for the size-biased measure}

In this section, we recall a well-known construction which allows for a simple representation of the size-biased measure, see for instance \cite[Lemma 1]{B04}.
We define $(\hat \go_i)_{i\ge 1}$ as a sequence of i.i.d. random variable with marginal distribution given by
\begin{equation}\label{tilted}
\hat \bbP[  \hat \go_1\in \cdot]= \bbE\left[ e^{\beta\go_{1,0}-\gl(\beta)}\ind_{\{\go_{1,0}\in \cdot\}}\right],
\end{equation}
 and $X$ a simple random walk ($\hat \bbP$ and $P$ denote the respective distributions). Note that the notations $\hat\omega$ and $\hat\bbP$ are not related to the point-to-point partition function $\hat W_n^\beta(x)$. Given $\go$, $\hat \go$ and $X$, all sampled independently, we define a new environment $\tilde \go=\tilde \go(X,\go,\hat \go)$  by
\begin{equation}
 \tilde \go_{i,x}\coloneqq \begin{cases}
	 \go_{i,x}  & \text{ if } x\ne X_i,\\
	 \hat \go_i  & \text{ if } x= X_i.
            \end{cases}
\end{equation}
In words,  $\tilde \go$ is obtained by tilting the distribution of the environment on the graph of $(i,X_i)_{i=1}^\infty$.
The following result states that the distribution of $\tilde \go$ corresponds to the size-biased measure.
\begin{lemma} \label{sbrepresent}It holds that
\begin{equation}
\tilde \bbP_n[\ (\go_{i,x})_{i\in \lint 1, n\rint, x\in \bbZ^d }\in \cdot \ ]=P\otimes \bbP\otimes \hat \bbP[\ (\tilde \go_{i,x})_{i\in \lint 1, n\rint, x\in \bbZ^d } \in \cdot ].
\end{equation}
\end{lemma}
We include a proof for completeness.
\begin{proof}
Given a bounded measurable function $f$ which depends only on the value of $\go$ up to time $n$ we have, by Fubini's theorem,
\begin{equation}
 \tilde \bbP_n[f(\go)]= \bbE\left[ W^{\beta}_{n}f(\go)\right]
 =E\left[  \bbE\left[e^{\beta H_n(\go,X)-\gl(\beta)} f(\go) \right]\right].
\end{equation}
Now for a fixed realization of $X$, $e^{\beta H_n(\go,X)-\gl(\beta)}$ is a probability density. Setting
$  \tilde \bbP_{n,X}(\dd \go)\coloneqq  e^{\beta H_n(\go,X)-\gl(\beta)}\bbP(\dd \go) $
we have
\begin{equation}
  \tilde \bbP_n[f(\go)]= E\left[ \tilde \bbE_{n,X}\left[ f(\go) \right]\right]
\end{equation}
and we can conclude by noticing that under  $\tilde \bbE_{n,X}$ the variables $\go_{i,x}$ are independent, the distribution of
$(\go_{i,X_i})_{i=1}^n$ is given by \eqref{tilted} while elsewhere the distribution of $\go_{i,x}$ is the same as under $\bbP$.
\end{proof}

\section{Organization of the proof of Theorem~\ref{main}}\label{orgamain}

\subsection{Identifying the right event}

As explained in Section~\ref{sizeb}, to prove that very strong disorder holds, we need to find an event which is typical under the size-biased measure $\tilde \bbP_n$ and atypical under the original measure $\bbP$.

\medskip

Our idea is to consider atypical fluctuations in the environment.
We look for regions in $\lint 1, n\rint\times \lint -n,n\rint^d$ (which is the area in which the size-biasing by $W^{\beta}_n$ has an effect) where the environment is especially favorable -- this corresponds to high values of $\go$ -- and show that these appear more frequently under the size-biased measure.

\medskip

Let us now be more precise  about the definition of favorable regions.
We fix an integer $s$. For $m\ge 0$ and $y\in \bbZ^d$, we say that $(m,y)$ is a \textit{good starting point} for $\go$ if $\theta_{m,y} W^{\beta}_{s} \ge n^{4d}$ (recall the definition of the shift operator introduced in \eqref{shift1}-\eqref{shift2}). Our goal is to show that
under $\tilde \bbP_n$ the probability to find a good starting point within $\lint 1, n-s\rint\times \lint -n,n\rint^d$ is very close to $1$ while it is close to zero under $\bbP$.
 We introduce thus the event
$$ A_n\coloneqq  \left\{\exists (y,m)\in \lint 1, n-s\rint\times \lint -n,n\rint^d, \ \theta_{m,y} W^{\beta}_{s} \ge n^{4d} \right\}.$$

\begin{proposition}\label{spot}
Assume that strong disorder holds. Then there 
exist $C>0$ and $n_0\in\N$ such that for all $n\ge n_0$ there exists  an $s=s_n\in  [0,C \log n]$ such that
\begin{equation}\label{bounds}\begin{split}
\bbP(A_n)\le \frac{1}{8 (2n+1)^{2d}}  \quad \text{ and } \quad  \tilde \bbP_n(A^{\complement}_n)\le \frac{1}{8 (2n+1)^{2d}}.
\end{split}\end{equation}

\end{proposition}

 The above immediately implies our main result.

\begin{proof}[Proof of Theorem~\ref{main} assuming Proposition~\ref{spot}]
Assume that strong disorder holds at $\beta$ and let $A_n$ be as in Proposition~\ref{spot} for $n$ sufficiently large. Using  Lemma~\ref{babac}, we have
\begin{equation}\label{upperlower}
 \bbE\left[ \sqrt{W_n}\right]\le \sqrt{\bbP(A_n)}+\sqrt{  \tilde \bbP_n(A^{\complement}_n)}\le \frac{1}{\sqrt{2}(2n+1)^d}
\end{equation}
and using Proposition~\ref{finitevol} we conclude that very strong disorder holds at $\beta$. This implies $\beta_c=\bar\beta_c$. For part (b), assume for contradiction that strong disorder holds at $\beta_c$, hence by the previous argument we must also have very strong disorder at $\beta_c=\bar\beta_c$. This contradicts part (b) of Theorem~\ref{thmx:CY}, so we must have weak disorder at $\beta_c$.
\end{proof}

Let us now make a couple of observations concerning Proposition~\ref{spot}. As a consequence of \eqref{fromabove} we have $W_n^\beta\leq L^{n}$, hence the choice of $s_n$ is optimal in the sense that the second bound in \eqref{upperlower} cannot hold for $s_n<c \log n$, where $c=1/(\log L)$.

\smallskip We also note that bounding $\bbP(A_n)$ is not  difficult (and the bound is valid for any choice of $s$).
Indeed, by translation invariance and Markov's inequality, we have
\begin{equation}\label{pan}
 \bbP(A_n)\le  \!\!\!\!\!\!\! \sum_
 {(m,y)\in \lint 1, n\rint\times \lint -n,n\rint^d}  \!\!\!\!\!\!\! \bbP\left( \theta_{m,y} W^{\beta}_{s} \ge n^{4d}\right)=  n (2n+1)^d
  \bbP\left( W^{\beta}_{s} \ge n^{4d}\right)\le n^{1-4d}(2n+1)^d.
\end{equation}

The main task in the proof of Proposition~\ref{spot} is thus to estimate $\tilde \bbP_n(A^{\complement}_n)$, which requires the combination of several ingredients and is performed in Section~\ref{proofspot}.
Let us provide here a rough sketch.
The main idea is to use the spine representation of Lemma~\ref{sbrepresent} to find for each $m$ a a random value $X_n$ such that $\tilde \bbP_n( \theta_{m,X_m} W^{\beta}_s\in \cdot )= \tilde \bbP_s( W^{\beta}_s\in \cdot)$.
To obtain a lower bound on  $\tilde \bbP_n(A^{\complement}_n)$, we then look at the event that $\{ \theta_{m,X_m} W^{\beta}_s\le n^{4d}\}$ for all $m$ which are multiples of $s$. The spacing by an amount $s$ ensures that $(\theta_{m,X_m} W^{\beta}_s)_{m\in (s\bbZ) \cap [0,n-s]}$ are independent under $\tilde \bbP_n$.
To conclude, we need a lower bound on $\tilde \bbP_s(W^{\beta}_{s} \ge n^{4d})$ which is provided by the following result.
\begin{proposition}\label{almostthesame}
If strong disorder holds then for any $\gep>0$ there exist $C(\gep),u_0(\gep)>0$ such that every $u\ge u_0$
 \begin{equation}\label{needed2}
 \exists s\in \lint 0,C \log u\rint , \quad \bbP\left[  W^{\beta}_s\ge u\right]\ge u^{-(1+\gep)}.
\end{equation}
 and as a consequence 
 \begin{equation}\label{needed3}
 \exists s\in \lint 0,C \log u\rint , \quad \tilde\bbP\left[  W^{\beta}_s\ge u\right]\ge u^{-\gep}.
\end{equation}
\end{proposition}
 As can be seen from above, the range of $s_n$ in Proposition~\ref{spot} comes from the fact that \eqref{needed3} is applied for $u=n^{4d}$.
While it is important in our reasoning that $s_n$ is not too large, we have a considerable margin of error and $s_n\le n^{\delta}$ for some small $\delta>0$ would be sufficient for our purpose.

\medskip

 Proposition~\ref{almostthesame}, contains the main technical difficulty in the paper and its proof is detailed in Sections~\ref{proofalsmostthesame},~\ref{intermezzo} and~\ref{discretos}.  We provide more insight concerning the organization of this proof in the next subsection.
 
 \subsection{Tail distribution of the maximum of the point-to-point partition function}\label{insight}

The proof of Proposition~\ref{almostthesame} relies on an estimate for the maximum of the point-to-point partition function over all possible endpoints, which we now present.
To motivate this result, we first notice that we can easily obtain, up to a constant factor, the tail distribution of the variable $\max_{n\ge 0}  W^{\beta}_n$.

\begin{lemma}\label{dull}
If strong disorder holds then for any $u\ge 1$,
\begin{equation}\label{dullbound}
 \bbP\left( \max_{n\ge 0}  W^{\beta}_n\ge u\right) \in\left[ \frac{1}{Lu},\frac{1}{u}\right].
 \end{equation}
\end{lemma}

\begin{proof}
We set $\tau_u\coloneqq  \inf\{ n\ge 0 \ \colon  \ W^{\beta}_n\ge u \}$ (with the convention $\inf \emptyset= \infty$). Note that from \eqref{fromabove} we have, on the event $\tau_u<\infty$,
\begin{equation}\label{step1}
W^{\beta}_{\tau_u}\in [u, Lu).
\end{equation}
Using the optional stopping theorem (first equality), and dominated convergence together with the assumption that $W^{\beta}_{\infty}=0$, we have
\begin{equation}\label{step2}
 1=\lim_{n\to \infty} \bbE\left[W^{\beta}_{\tau_u\wedge n} \right]=\bbE\left[W^{\beta}_{\tau_u}\ind_{\{\tau_u<\infty\}} \right].
\end{equation}
The result follows from the combination of \eqref{step1} and \eqref{step2}.
\end{proof}

In view of the desired result \eqref{needed2}, the task is thus to get a bound on the time when $(W_n^\beta)_{n\in\N}$ achieves its maximum. 
The following theorem is the key step to obtaining this control. It states that
the maximum of \textit{point-to-point} partition functions has the same tail distribution (up to a constant factor) as $\max_{n\ge 0} W^{\beta}_n$. This result is of independent interest and is the most technical part of the proof.

\begin{theorem}\label{locaend}
If strong disorder holds then there exists a constant $c>0$ such that for all $u\ge 1$
\begin{equation}\label{locatail}
 \bbP\left[ \max_{n\ge 0, x\in \bbZ^d} \hat  W^{\beta}_n(x)\ge u\right] \ge \frac{c}{u}.
 \end{equation}
\end{theorem}

To close this section, we  further mention that under the very strong disorder assumption, the maximum of  $W^{\beta}_n(x)$ is with large probability attained for a ``small'' value of $n$. More precisely, we have the following  consequence of Lemma~\ref{dull} and Theorem~\ref{locaend}.

\begin{corollary}\label{notessential}
If very strong disorder holds then there exist constants $c,C>0$ which are such that for all $u\ge 1$,
 \begin{align}
	 \bbP\left[ \max_{(n,x)\in \llbracket 0, C\log u\rrbracket \times\bbZ^d} \hat  W^{\beta}_n(x)\ge u\right] &\ge \frac{c}{u},\\
	 \quad \bbP\left[  \exists s\in \lint 0,C \log u\rint , \    W^{\beta}_s\ge u\right]&\ge \frac{c}{u}.
\end{align}

\end{corollary}

We note that \eqref{fromabove} is not used in the proof of this corollary, only the weaker assumption \eqref{expomoment}. Moreover, the above corollary is not used to prove the main result. We state it rather to highlight that
 once Theorem~\ref{main} is established, a stronger variant of Proposition~\ref{almostthesame} (with no need for $\gep>0$ in the exponent) can be proved rather easily.
Proposition~\ref{almostthesame} is thus not an optimal result but rather a (crucial) technical step in our proof.
The proof of Corollary~\ref{notessential} is presented in Appendix~\ref{ouix}.

\subsection{Organization of the remaining sections}
Let us outline how the remainder of the paper is organized:
\begin{itemize}
 \item
In Section~\ref{proofspot} we prove Proposition~\ref{spot} assuming Proposition~\ref{almostthesame}.
\item
In Section~\ref{proofalsmostthesame} we show how to deduce  Proposition~\ref{almostthesame} from Theorem~\ref{locaend}.
\end{itemize}
The proof of Theorem~\ref{locaend} relies on a series of arguments which include martingale techniques: Doob's martingale decomposition, martingale concentration estimates, etc...

\smallskip In order to help the reader, we first provide the proof in a discrete-space, continuous-time setup: the directed polymer in a Brownian environment \cite{CH06}.
The computations for continuous martingales are cleaner than the ones performed in discrete time, and we believe this allows for a clearer exposition of the main ideas behind the computation. Section~\ref{intermezzo} is devoted to this technical intermezzo. 
Finally in Section~\ref{discretos}, we provide the proof of Theorem~\ref{locaend}, paralleling the ideas introduced in Section~\ref{intermezzo}.

\section{Proof of Proposition~\ref{spot}}\label{proofspot}

Recall from Equation \eqref{pan} that the required bound on $\bbP(A_n)$ is valid for any choice of $s$. We can thus focus on bounding $\tilde \bbP(A^{\complement}_n)$ assuming that  Proposition~\ref{almostthesame} holds.
We apply Proposition~\ref{almostthesame} for $u=n^{4d}$  and $\gep=1/(12d)$ and we consider $s\in \lint 0, 4C \log n\rint$, which is such that 
\begin{equation}\label{sbes}
 \tilde \bbP_s(W^\beta_s\ge n^{4d})  \ge n^{-4d\gep}=n^{-1/3}
\end{equation}
(here the value $1/3$ is chosen rather arbitrarily and could be replaced by any number in $(0,1)$).
 Using Lemma~\ref{sbrepresent} (and the fact that the event $A_n$ is $\cF_n$-measurable), we have
\begin{equation}\label{gjiks}
 \tilde \bbP_n(\go \in A^{\complement}_n) = P\otimes \bbP \otimes \hat \bbP\left(  \tilde\go \in A^{\complement}_n  \right).
\end{equation}
We let  $\theta_{m,y}\tilde W^{\beta}_{s}$ denote the partition function constructed with the environment $\theta_{m,y} \tilde \go:= (\tilde{\go}_{i+m,x+y})_{i\ge 0,x\in \bbZ^d}$. From the definition of $A_n$ and the fact that $|X_m|\le m$, it follows immediately that
\begin{multline}\label{gjikss}
 \{ \tilde \go  \in A^{\complement}_n\}=  \left\{ \forall (m,y)\in \lint0,n-s\rint \times \lint -n,n\rint^d \ \theta_{m,y}\tilde W_s< n^{4d}  \right\} \\  \subseteq \left\{ \forall j\in \left\lint0, \left\lfloor  \frac{n}{s}-1\right\rfloor\right\rint, \ \theta_{j,X_{js}}\tilde W_s< n^{4d}  \right\}.
\end{multline}
Now the important observation is the following.
\begin{lemma}\label{cruxial}
 Under  $P\otimes \bbP \otimes \hat \bbP$, the variables $(\theta_{js,X_{js}}\tilde W_s)_{j\ge 0}$ are independent and identically distributed.
\end{lemma}
\begin{proof}
 Letting $\tilde \go^{(j)}:=(\tilde{\go}_{i+js, x+ X_{js}})_{i\in \lint 1,s\rint,x\in \bbZ^d}$ denote the shifted environment restricted to the time slice $\lint 1,s\rint$. Since $(\theta_{js,X_{js}}\tilde W_s)$ is a function of $\tilde \go^{(j)}$, it is sufficient to show that $(\tilde \go^{(j)})_{j\ge 0}$ are independent and identically distributed.
 
\medskip

\begin{figure}
 \begin{center}
%  \leavevmode
% \epsfxsize =12 cm
% \epsfbox{spine2.eps}
\includegraphics[width=.7\textwidth]{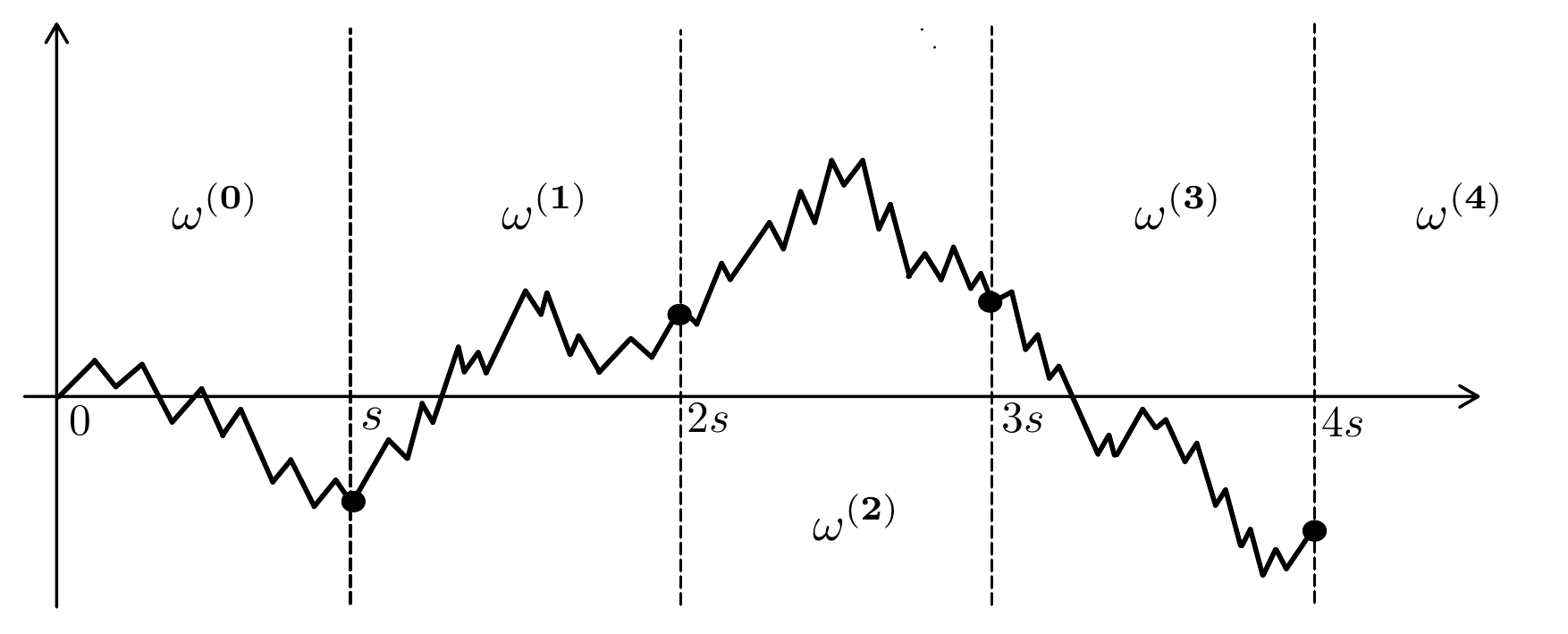}
\end{center}
\caption{\label{lafigure}\footnotesize{
The environment $\go^{(j)}$ is simply the environment restricted to the $j$-th slice of width $s$, which has been shifted so that $(js,X_{js})$ (the solid circles on the pictures) plays the role of the origin. Since the slices are disjoint and $X$ is independent of $\go$, the environments $(\go^{(j)})_{j\ge 1}$ are independent and distributed like $\go^{(0)}$.
The environments $\tilde \go^{(j)}$ are obtained by replacing the environment by $\hat \go$ along the spine. The Markov property applied to $X$ ensures that $(X^{(j)})_{j\ge 0}$ (which described the spine portions after shift) are i.i.d.\ and the same thing can be said about the restriction of $\hat \go$ to each time interval $\lint js+1,(j+1)s\rint$.}
}
\end{figure}

Let us consider 
$$X^{(j)}= (X_{js+i}-X_{js})_{i\in \lint 1,s\rint}, \quad   \go^{(j)}=({\go}_{i+js, x+ X_{js}})_{i\in \lint 1,s\rint,x\in \bbZ^d} \quad \text{ and } \quad \hat \go^{(j)}=(\hat \go_{js+n})_{n\in \lint 1,s\rint}.$$
The environment $\tilde \go^{(j)}$ is a function of the triplet $(X^{(j)},\go^{(j)},\hat \go^{(j)})$. More precisely we have
$$\tilde \go^{(j)}_{i,x}= \go^{(j)}_{i,x}\ind_{\{X^{(j)}_i\ne x\}}+  \hat \go^{(j)}_i \ind_{\{X^{(j)}_i=x\}}. $$
Furthermore, under $P\otimes \bbP \otimes \hat \bbP$, $(X^{(j)}, \go^{(j)}, \hat \go^{(j)})_{j\ge 0}$ is an i.i.d.\ sequence (see Figure~\ref{lafigure})
yielding the desired conclusion.
\end{proof}
Combining \eqref{gjiks} and \eqref{gjikss} we have 
\begin{equation}\begin{split}
  \tilde \bbP_n(A^{\complement}_n)&\le P\otimes \bbP \otimes \hat \bbP \left( \forall j\in \left\lint0, \left\lfloor  \frac{n}{s}-1\right\rfloor\right\rint, \ \theta_{j,X_{js}}\tilde W_s< n^{4d}  \right)\\
  &= P\otimes \bbP \otimes \hat \bbP \left( \tilde W_s< n^{4d}  \right)^{\lfloor n/s\rfloor}\\
  &= \tilde \bbP_s\left(W_s<n^{4d}\right)^{\lfloor n/s\rfloor}
\end{split}\end{equation}
where the first equality above is from Lemma~\ref{cruxial} and the second one from Lemma~\ref{sbrepresent}.
Using \eqref{sbes} and the fact that $s\le C\log n$ we conclude that $\tilde \bbP_n(A^{\complement}_n)\le \exp(-\sqrt{n})$.
\qed

\section{Proof of Proposition~\ref{almostthesame}}\label{proofalsmostthesame}

In this section, we show how Proposition~\ref{almostthesame} can be deduced from Theorem~\ref{locaend}.
Let us first explain how  \eqref{needed3} follows from \eqref{needed2}. We have
\begin{equation}
\tilde \bbP_s( W^{\beta}_{s} \ge u  )=   \bbE\left[ W^{\beta}_{s}\ind_{\{W^{\beta}_{s} \ge u\}}  \right]\ge  u  \bbP( W^{\beta}_{s} \ge u),
\end{equation}
and hence $\bbP( W^{\beta}_{s} \ge u)\ge u^{-1-\gep}$ implies that $\tilde \bbP_s( W^{\beta}_{s} \ge u  )\ge u^{-\gep}$.
Concerning \eqref{needed2}, we observe that by a union bound we have
\begin{equation}\begin{split}
 \max_{n\in \lint 0,C \log u\rint }  \bbP\!\left[  W^{\beta}_n\ge u\right] &\ge (C\log u)^{-1}  \bbP
 \left[ \max_{n\in \lint 0,C \log u\rint } W^{\beta}_n\ge u\right] \\ &\ge (C\log u)^{-1} \bbP\left[ \max_{(n,x)\in \lint 0,C \log u\rint \times \bbZ^d} \hat W^{\beta}_n(x)\ge u\right].
\end{split}\end{equation}
Hence \eqref{needed2} can be deduced immediately from the following stronger result, which we prove with the help of Theorem~\ref{locaend}.
\begin{proposition}\label{neededresult}
Assume that strong disorder holds.
Given $\gep>0$, there exist constants $C=\ C(\gep)$ and $u_0=\ u_0(\gep)>0$ such that for all $u\ge u_0$,
\begin{equation}\label{needed}
 \bbP\left[ \max_{(n,x)\in \lint 0,C \log u\rint \times \bbZ^d} \hat W^{\beta}_n(x)\ge u\right]\ge u^{-(1+\gep)}.
\end{equation}
\end{proposition}

 \begin{proof}
Using Theorem~\ref{locaend}, for any fixed $v$, we can find $T=T(v)$ such that
\begin{align}\label{indstart}
\bbP\left[ \max_{(n,x)\in \lint 0,T\rint \times \bbZ^d} \hat W_n(x)>v\right]\ge \frac{c}{2v}.
\end{align}
We are going to prove by induction that for every $k\ge 1$,
\begin{equation}\label{recurs}
\bbP\left[ \max_{(n,x)\in \lint 0,kT\rint \times \bbZ^d} \hat W_n(x)>v^k\right]\ge \left(\frac{c}{2v}\right)^{k}.
\end{equation}
Then \eqref{needed} can be deduced from \eqref{recurs} by choosing $v=v(\eps)$ such that $2 v^{-\gep/2}\le c$ and setting $C(\eps)= 2T(v)/(\log v)$. Indeed, for $u\geq u_0$ we have $k\coloneqq \lceil\log_v(u) \rceil\leq 2\log(u)/\log(v)$ and $v^{-(1+\frac{\gep}{2})}\geq u^{-\frac{\gep}{2}}$. Thus, by \eqref{recurs},
\begin{align*}
	\bbP\left[ \max_{(n,x)\in \lint 0,C \log u\rint \times \bbZ^d} \hat W^{\beta}_n(x)\ge u\right]\ge \bbP\left[ \max_{(n,x)\in \lint 0,kT\rint \times \bbZ^d} \hat W_n(x)>v^k\right]\geq v^{-k(1+\frac{\gep}{2})}\geq u^{-1-\gep}.
\end{align*}

Let us now prove \eqref{recurs} by induction. The case $k=1$ in \eqref{recurs} is simply \eqref{indstart}. For the induction step,  we let $A_k$ denote the event in \eqref{recurs} and define
$$(\sigma,Z)\coloneqq \min \left\{ (n,x)\in \N\times\Z^d \ \colon  \   \hat W_n(x)>v^k  \right \} $$
where $\min$ refers to the minimal element of the set for  the lexicographical order on $\N\times\Z^d$. By convention  $(\sigma,Z)=(\infty,0)$ if the set is empty. Recalling \eqref{shift2},
\begin{equation}
 \hat W_{n+m}(x+z)\ge \bbE\left[ e^{\beta H_{n+m}(X)-(n+m)\gl(\beta)}\ind_{\{X_m=z, X_{n+m}=x+z\}}\right]= \hat W_{m}(z) \theta_{m,z}  \hat W_{n}(x).
 \end{equation}
Hence  $\hat W_{n+m}(x+z)\ge v^{k+1}$ if $\hat W_{m}(z)\ge v^k$ and  $\theta_{m,z}\hat W_{n}(x)\ge v$.
Decomposing over the possible values of $(\sigma,Z))$ (recall \eqref{shift2}) we  have
\begin{equation}
\bigcup_{m\le kT , z\in  \bbZ^d}\big( \{(\sigma,Z)=(m,z) \}\cap \theta_{m,z} A_1 \big)\subset A_{k+1}.
\end{equation}
Furthermore for every $m$ and $z$ the events $(\sigma,Z)=(m,z)$ and $\theta_{m,z}A_1$ are independent since their realizations depend on disjoint regions of an i.i.d. environment.
Hence we have
\begin{equation}\begin{split}
 \bbP\left[A_{k+1}\right]&\ge \sumtwo{m\le kT}{z\in \bbZ^d} \bbP\left[ \{(\sigma,Z)=(m,z)\} \cap\theta_{m,z} A_1  \right],\\
&=\sumtwo{m\le kT}{z\in \bbZ^d}  \bbP\left[ (\sigma,Z)=(m,z)\right] \bbP\left[ \theta_{m,z} A_1  \right]= \bbP[A_{k}]\bbP[A_1],
\end{split}\end{equation}
where the last equality uses that $\bbP\left[ \theta_{m,z}A_1 \right]=\bbP[A_1]$ and that  $\{(\sigma,Z)=(m,z)\}_{m\le kT, z\in \bbZ^d}$ is a partition of the event $A_k$.
\end{proof}

\section{Intermezzo: Tail estimate for the partition function in a continuum setup } \label{intermezzo}

\subsection{Model and result}

We consider in this section a continuous-time variant of the directed polymer model. We let  $Y=(Y_t)_{t\ge 0}$ denote the continuous-time simple random walk on $\bbZ^d$, i.e., $Y$ has the same law as $(X_{N_t})_{t\geq 0}$ where $X$ is as before and $(N_t)_{t\geq 0}$ is an independent, homogeneous Poisson process of rate $1$. The law $Y$ is denoted by $\bP$. The environment $\eta=(\eta^{(x)}_t)_{x\in \bbZ^d, t\ge 0}$ is given by a collection of independent standard Brownian motions indexed by $\bbZ^{d}$, and we use the letter $\bbP$ for the associated probability distribution. Given a fixed realization of $\eta$ and $t\ge 0$, we set
\begin{equation}
\bP_{\eta,t}^\beta(\dd Y)\coloneqq\frac{1}{Z_t^\beta}e^{\beta \bH_t(\eta,Y)-\frac{\beta^2}{2}t}\bP(\dd Y),
\end{equation}
where the Hamiltonian $\bH_t(\eta,Y):= \sum_{x\in \bbZ^d}\int^t_{0} \ind_{\{Y_t=x\}}\dd \eta^{(x)}_t$ is  defined by integrating the environment along the trajectory $Y$ and 
  $Z_t^\beta:= \bE\big[ e^{\beta \bH_t(\eta,Y)-\frac{\beta^2}{2}t}\big]$.
We let $(\mathcal F_t)_{t\ge 0}$ denote the natural filtration associated with the environment $\eta$.
 As in the discrete case, one can check that $Z^{\beta}_t$ is a martingale with respect to this filtration.
We further define the point-to-point partition function by
\begin{equation}
 \hat Z_t^\beta(x)\coloneqq \bE\big[e^{\beta \bH_t(\eta,Y)-\frac{t\beta^2}{2}}\ind_{\{Y_t=x\}}\big].
\end{equation}
We assume that $\beta$ is such that \textit{strong disorder} holds, that is to say that  $\lim_{t\to \infty } Z^{\beta}_t=0$. The analogue of Theorem~\ref{b2bcrit} is also valid in that context (it is proved in \cite{BS10} for $d\ge 4$ and in \cite{BS11} for $d=3$) meaning that strong disorder implies that
\begin{equation}\label{strictl2}
\beta > \beta_2= \frac{1}{\sqrt{\int^\infty_0\bP( Y^{(1)}_t=   Y^{(2)}_t)}}.
\end{equation}
The main object of this section is to prove the following analogue of Theorem~\ref{locaend}.
\begin{theorem}\label{locaend2}
If strong disorder holds then there exists a constant $c>0$ such that for all $u\ge 1$
\begin{equation}
 \bbP\left[ \max_{t\ge 0, x\in \bbZ^d} \hat  Z^{\beta}_t(x)\ge u\right] \ge \frac{c}{u}.
 \end{equation}
\end{theorem}

\begin{remark}\label{remdrop}
We could in principle go further and also prove the analogue of Theorem~\ref{main} for the directed polymer in a Brownian environment. This would require some minor adaptation of the argument since, for instance, the proofs of Proposition~\ref{finitevol} and Proposition~\ref{spot} rely on the fact that $X$ cannot walk a distance greater than $n$ in $n$ steps ( one possibility for this would be to follow the approach used in  \cite{JL25}). This goes beyond our scope. Our sole reason to present the proof of Theorem~\ref{locaend2} is to expose more clearly the mechanism at work in the proof of Theorem~\ref{locaend}.
\end{remark}

\subsection{Proof strategy for Theorem~\ref{locaend2}}

Before getting to the details of the proof, we give an overview of the argument for the proof of Theorem~\ref{locaend2}. We denote the first hitting time of $u$ by
\begin{equation}\label{tauucont}
 \tau_u\coloneqq \inf\left\{ t>0 \ \colon  \ Z_t^\beta=u\right\}.
 \end{equation}
 The strong disorder assumption 
implies that (see the proof of Lemma~\ref{dull})
\begin{equation}\label{simple}
\bbP(\tau_u<\infty)=1/u.
\end{equation}
We want to show that conditional on $\tau_u<\infty$, with probability larger than $\delta>0$ (where $\delta$  does not depend on $u$)  we can find a random time $\cT>\tau_u$ which is such that two things occur:
\begin{itemize}
 \item [(A)] The end point is localized, that is, $\max_{x\ge 0}\bP^{\beta}_{\eta,\cT}(X_\cT=x)\ge \delta$.
 \item [(B)] $Z^{\beta}_t$ is not too small, or more specifically $Z^{\beta}_{\cT}\ge \delta u$.
\end{itemize}
Note that $(A)$ and $(B)$ combined imply that
$ \max_{x\in\bbZ^d} \hat Z^{\beta}_{{\cT}} (x)\ge \delta^2 u$. Hence, if we show that the two conditions can be satisfied with probability $\delta$, we have
\begin{equation}
 \bbP\left(\max_{t\ge 0,x\in\bbZ^d} \hat Z^{\beta}_{t}(x)\ge \delta^2 u\right) \ge \delta \bbP(\tau_u<\infty)= \delta u^{-1},
\end{equation}
which is equivalent to Theorem~\ref{locaend2}.

\subsubsection*{Connection between the partition function and endpoint replica overlap}
It is convenient at this point to introduce the analogs of \eqref{endpoint} and \eqref{overlap},
\begin{equation}\label{def_It}
 \mu^{\beta}_{t}(x)\coloneqq \bP^{\beta}_{\eta,t}(Y_t=x) \quad  \text{ and } \quad
I_t\coloneqq  \sum_{x\in \bbZ^d} \mu^{\beta}_{t} (x)^2.
\end{equation}
The quantity $I_t$ is naturally associated with the $\log$-partition function $\log Z^{\beta}_t$. More precisely, from Itô's formula we have
\begin{equation}\label{eq:ito}
 \log Z^{\beta}_t= M_t-\frac{\beta^2}{2} \int^t_0 I_s \dd s,
 \end{equation}
where $(M_t)_{t\ge 0}$ is the martingale defined by $ M_t=\int^t_0\frac{\dd Z^{\beta}_s}{Z^{\beta}_s}.$
The quantity $I_t$ also appears a second time when computing the quadratic variation of $\log Z^{\beta}_t$ (which is also that of $M_t$)
\begin{equation} \langle M \rangle_t=   \beta^2 \int^t_0 I_s \dd s.
 \end{equation}
 The above connection between the $\log$-partition function and the replica overlap for the endpoint is not specific to the Brownian polymer, it also holds (although in a weaker form) for the discrete model. It has been extensively used to prove localization results both in the very strong disorder phase \cite{BC20,CH02,CSY03} and in the strong disorder phase \cite{CH06,Y10}. The reasoning we develop to find a localization time -- that is, to find $t>\tau_u$ for which $(A)$ and $(B)$ are satisfied -- relies on this connection and can be decomposed into three steps.

 \subsubsection*{Step 1: The integrated endpoint replica overlap is large}

Fix some parameter $K>1$ and observe from \eqref{simple} that if $\tau_u<\infty$ then with conditional probability $1/K$ we also have $\tau_{Ku}<\infty$ (our choice of $K$ will be such that $1/K$ is larger than $\delta$).
Within the time interval $(\tau_u,\tau_{Ku}]$, the supermartingale $\log Z^{\beta}_t$ increases by an amount $\log K$, so due to \eqref{eq:ito} we must have
\begin{equation}
	M_{\tau_{Ku}}-M_{\tau_u}\geq \log K.
\end{equation}
Using stochastic calculus, we can show that this implies that with probability larger than  $\frac1K-\frac1{K^2}$ the increment of the quadratic variation  $\langle M\rangle_{\tau_{Ku}}-\langle M\rangle_{\tau_u}$ is larger than $(\log K)/4$. This is done in Proposition~\ref{keyprop1}.

 \subsubsection*{Step 2: $I_s$ cannot be small everywhere on $(\tau_u,\tau_{Ku}]$}
The hard part is to show that if $\int^{\tau_{Ku}}_{\tau_u} I_t\ \dd t$ is large then, most likely, there exists a time $t\in (\tau_u,\tau_{Ku}]$ such that $I_t\ge \delta$. For this latter step, we borrow tools and ideas which were originally used to prove that endpoint localization occurs infinitely often in the strong disorder regime. They were introduced in \cite{CH02,CH06}. In  \cite{Y10}
the argument was adapted to a more general setup with an improved presentation. The idea to use these techniques to derive upper tail bounds has appeared in \cite{J22} in the weak disorder regime. 
Both the original argument presented in \cite{CH02,CH06,Y10} and our elaboration on it are very technical. 

\smallskip Let us try to give a rough idea of the argument. It relies on the study of  an \textit{ad-hoc} continuous process $(J_t)_{t\ge 0}$ (introduced in a discrete setup in \cite[Equation (3.3)]{Y10} as $(X_t)_{t\ge 0}$). The quantity $J_t$ is obtained by taking the scalar product between  $\mu^{\beta}_t$ and $G \mu^{\beta}_t$ where $G$ is the Green function associated to the random walk.
The key point of the arguments are that:
\begin{itemize}
 \item [(i)] The process $(J_t)_{t\ge 0}$ is bounded,
 \item [(ii)] Whenever $I_t\le \delta$, the \textit{drift} of $J_t$ (the derivative of the finite variation part in the semimartingale decomposition of $J$) is positive and proportional to $I_t$.
\end{itemize}
The specific choice we make for $J_t$ is crucial to make property $(ii)$ valid. Another requirement for the proof is that we must have $\beta\ne \beta_2$ (which by Theorem \ref{b2bcrit} holds whenever strong disorder holds). We provide more detailed explanation on this specific point after the proof in Section~\ref{expl}.

\smallskip Now, if  $\int^{\tau_{Ku}}_{\tau_u} I_t\ \dd t$ is large  and $I_t\le \delta$ on the interval $(\tau_u,\tau_{Ku}]$, $(ii)$ implies that the integrated drift of $J_t$ over the interval  $(\tau_u,\tau_{Ku}]$ is also large. On the other hand, the probability that the martingale part of $J_t$ ``compensates'' this drift is small, and  thus we obtain a contradiction with the fact that the process is bounded.

 \subsubsection*{Step 3: There is no ``dip'' between $\tau_u$ and $\tau_{Ku}$}
The combination of the first two steps allows to find $t\in (\tau_u,\tau_{Ku}]$ that satisfies $(A)$.
To check that it also satisfies $(B)$, we show that the probability that $Z^{\beta}_t$ drops below level $u/K$ within the interval 
$(\tau_u,\tau_{Ku}]$ is small. This is proved in Lemma~\ref{keylemma}.

\subsection{Organization for the proof of Theorem~\ref{locaend2}}

 We introduce the short-hand notation $\bbP_u=\bbP\left[ \ \cdot \   | \ \cF_{\tau_u}\right]$.
The following two results correspond to the first two steps of the argument.

 \begin{proposition}\label{keyprop1}
If strong disorder holds, then for any $K>1$ we have
 \begin{equation}\label{rrrr}
    \bbP_u\left( \tau_{Ku}<\infty \ ; \ \langle M \rangle_{\tau_{Ku}}-\langle M \rangle_{\tau_{u}}<\frac{\log K}{4}\right)\le \frac{1}{K^2}
 \end{equation}
on the event $\{\tau_u<\infty\}$.
 \end{proposition}

\begin{proposition}\label{keyprop2}
 If strong disorder holds and $\delta$ is sufficiently small, then 
 \begin{equation}
  \bbP_u\left(  \langle M \rangle_{\tau_{Ku}}-\langle M \rangle_{\tau_{u}}\ge\frac{\log K}{4}\ ; \ \max_{t\in(\tau_u,\tau_{Ku}]  }
  I_t\le \delta \right)\le \frac{1}{K^2}
 \end{equation}
 on the event $\{\tau_u<\infty\}$.
\end{proposition}

Now for the third step, we set $
  \sigma_{u,K}\coloneqq  \inf\{ t\ge \tau_u \ \colon  \  Z^{\beta}_t\le u/K   \}.$
\begin{lemma}\label{keylemma}
If strong disorder holds, then on the event $\{\tau_u<\infty\}$
\begin{equation}
 \bbP_u( \sigma_{u,K}<\tau_{Ku}<\infty)=\frac{1}{K(K+1)}.
\end{equation}
\end{lemma}

Proposition~\ref{keyprop1} and Lemma~\ref{keylemma} have short and direct proofs which we include at the end of this subsection. 
The proof Proposition~\ref{keyprop2} on the other hand requires considerable work and is given in Section~\ref{seckeyprop}.
Let now show how the combination of the three results yields the main theorem.

\begin{proof}[Proof of Theorem~\ref{locaend2}]

We use the following inclusion

\begin{equation}\label{incluclu} \left\{  \tau_{Ku}<\sigma_{u,K} \ ;  \langle M \rangle_{\tau_{Ku}}-\langle M \rangle_{\tau_{u}} \ge \frac{\log K}{4}\ ; \ \max_{t\in(\tau_u,\tau_{Ku}]  }
  I_t> \delta  \right\}
  \subset \left\{ \maxtwo{t\ge 0}{x\in \bbZ^d} \hat Z^{\beta}_t(x)\ge \frac{\delta u}{K} \right\}.
  \end{equation}
  Indeed, if we define $T\coloneqq \inf\{t>\tau_u \ \colon  \ I_t\ge \delta \}$, then on the first event we have $T\le \sigma_{u,K}$. Hence
  $Z^{\beta}_{T}\ge u /K$ and, recalling \eqref{def_It},
  \begin{equation}
\max_{x\in \bbZ^d} \hat Z^{\beta}_{T}(x)= Z^{\beta}_{T} \max_{x\in \bbZ^d}\mu^{\beta}_{T}(Y_T=x)\ge Z^{\beta}_{T} I_{T}\ge \frac{\delta u}{K}.
\end{equation}
Now, using Propositions~\ref{keyprop1} and~\ref{keyprop2} as well as Lemma~\ref{keylemma}, we have
\begin{multline}\label{frombelow}
  \bbP_u\left(  \tau_{Ku}<\sigma_{u,K} \ ; \langle M \rangle_{\tau_{Ku}}-\langle M \rangle_{\tau_{u}}\ge \frac{\log K}{4}\ ; \ \max_{t\in(\tau_u,\tau_{Ku}]  }
  I_t> \delta  \right)\\ \ge   \bbP_u(  \tau_{Ku}<\infty) -
  \bbP_u\left( \sigma_{u,K}<\tau_{Ku}<\infty \right)-  \bbP_u\left( \tau_{Ku}<\infty \ ; \   \langle M \rangle_{\tau_{Ku}}-\langle M \rangle_{\tau_{u}} <\frac{\log K}{4}\right)\\-\bbP_u\left(   \langle M \rangle_{\tau_{Ku}}-\langle M \rangle_{\tau_{u}}\ge  \frac{\log K}{4} \ ; \ \max_{t\in(\tau_u,\tau_{Ku}]  }
  I_t\le \delta\right) \ge \frac{1}{K}-\frac{3}{K^2}.
\end{multline}
Taking expectation over the event $\{\tau_u<\infty\}$ (recall \eqref{simple})
we obtain
\begin{equation}
 \bbP\left( \maxtwo{t\ge 0}{x\in \bbZ^d} \hat Z^{\beta}_t(x)\ge \frac{\delta u}{K}\right)\ge \frac{1}{u}\left(\frac{1}{K}-\frac{3}{K^2}\right).\tag*{\qedhere}
\end{equation}

\end{proof}

The proof of Proposition~\ref{keyprop1} boils down to applying the following classical estimate for which we include a proof.
\begin{lemma}\label{largedevcont}
 Let $N_t$ be a continuous-time martingale such that $t\mapsto N_t$ is almost surely continuous, starting at $0$. Let $v>0$ and $T_v$ be the first time that
 $N_t$ reaches $v$.
For any $a>0$,
 \begin{equation}
  \bbP[ \langle N\rangle_{T_v}\le  a \ ; \ T_v<\infty ] \le e^{-\frac{v^2}{2a}}.
 \end{equation}
In particular, for any stopping time $T'$, 
\begin{equation}
	\bbP[ \langle N\rangle_{T'}\le  a \ ; \ N_{T'}\geq v\ ; \ T'<\infty ] \le e^{-\frac{v^2}{2a}}.
\end{equation}
\end{lemma}

\begin{proof}
Using the stochastic exponential of the martingale $N_{T_v\wedge t}$, we have
\begin{equation}
1=\bbE\left[ e^{ \frac{v}{a} N_{T_v\wedge t}- \frac{v^2}{2a^2}\langle N\rangle_{T_v\wedge t} }\right].
\end{equation}
Letting $t\to \infty$, we have
$$1\ge \bbE\left[ e^{ \frac{v}{a} N_{T_v}- \frac{v^2}{2a^2} \langle N\rangle_{T_v} }\ind_{\{T_v<\infty\}}  \right]
\ge e^{\frac{v^2}{2a}}  \bbP\left[ \ \langle N\rangle_{T_v}\le a \ ; \ T_v<\infty \ \right],$$
which yields the desired conclusion. For the second claim, we to note that $T_v\leq T'$ on the event $N_{T'}\geq v$, hence $\langle N\rangle_{T_v}\leq \langle N\rangle_{T'}$.
\end{proof}

\begin{proof}[Proof of Proposition~\ref{keyprop1}]
 We consider the martingale $(M_{s+\tau_u}-M_{\tau_u})$ with the filtration $(\cF_{s+\tau_u})$, and apply Lemma~\ref{largedevcont}  with 
 $v=\log K$, $a=\log K/4$  and $T_v=\inf\{s\geq 0\colon  M_{s+\tau_u}>M_{\tau_u}+v\}$.
 We obtain 
 \begin{equation}\label{rrrrr}
   \bbP_u\left( T_{v}<\infty \ ; \ \langle M \rangle_{T_v+\tau_u}-\langle M \rangle_{\tau_{u}}<\frac{\log K}{4}\right)\le K^{-2}.
 \end{equation}
Strictly speaking, Lemma~\ref{largedevcont} yields the result for $\bbP$ but its proof can be replicated \textit{verbatim} to obtain the inequality for the conditional expectation $\bbP_u$.
Now, due to \eqref{eq:ito} we have 
\begin{equation}
M_{\tau_{Ku}}-M_{\tau_{u}}\ge \log Z_{\tau_{Ku}}- \log Z_{\tau_u}=\log K
 \end{equation}
 and hence $T_v+\tau_u\le \tau_{Ku}$. Thus \eqref{rrrrr} implies \eqref{rrrr}.
\end{proof}

\begin{proof}[Proof of Lemma~\ref{keylemma}]
We consider the martingale $(Z^{\beta}_{s+\tau_u})$ with the filtration $(\cF_{s+\tau_u})$. 
Applying the optional stopping theorem at  time $(\sigma_{u,K}\wedge \tau_{Ku})-\tau_u$ we have
\begin{equation}
  \bbP_u( \sigma_{u,K}<\tau_{Ku})=\frac{K-1}{K-(1/K)}=\frac{K}{K+1}.
\end{equation}
We have used that $\lim_{t\to\infty}Z_t=0$. In the same manner, considering the martingale $Z^{\beta}_{s+\sigma_{u,K}}$ on the event  $\{\sigma_{u,K}<\tau_{Ku}\}$ and applying  the optional stopping theorem at $\tau_{Ku}$, we obtain
\begin{equation}
   \bbP_u( \tau_{Ku}<\infty  \ | \ \sigma_{u,K}<\tau_{Ku} )=\frac{1}{K^2},
\end{equation}
which allows to conclude.
\end{proof}

\subsection{Proof of Proposition~\ref{keyprop2}}\label{seckeyprop}

We define $G$ to be Green's operator associated with the random walk obtained by taking the difference of two independent copies of $Y$. We set
\begin{equation}\label{defg}
 g(x)=\int^{\infty}_0 \bP^{\otimes 2}(Y^{(1)}_t-Y^{(2)}_t=x)\dd t \quad  \text{ and } \quad  G(x,y)=g(x-y).
\end{equation}
Note that  $G=(-2\Delta)^{-1}$, where $\Delta$ is the discrete Laplace operator on $\bbZ^d$, $\Delta=D-\mathrm{id}$ where $\mathrm id$ denotes the identity. Note also that $2\Delta$ is the generator of the walk $(Y^{(1)}_t-Y^{(2)}_t)_{t\ge 0}$.
Furthermore, under the strong disorder assumption, by \eqref{strictl2} we have
\begin{align}\label{g0positive}
	\beta^2g(0)>1.
\end{align}

Given a function $a\colon \bbZ^d\to \bbZ$, the function $Ga$ is defined by
setting $Ga(x)\coloneqq \sum_{y\in \bbZ^d} G(x,y)a(y)$ when the sum is well defined.
If $a$ and $b$ are functions defined on $\bbZ^d$ such that
 $\sum_{x\in \bbZ^d}|a(x)b(x)|<\infty$
we use the notation $(a, b)$ to denote the scalar product
\begin{equation}\label{scalarprod}
(a, b)=\sum_{x\in \bbZ^d}a(x)b(x).
 \end{equation}
We extend our use of the notation \eqref{scalarprod} to the case where one of the terms is the infinitesimal variation of a stochastic process.
We consider the process $(J_t)_{t\ge 0}$ defined by
\begin{equation}\label{deffjt}
 J_t\coloneqq ( \mu^{\beta}_t , G \mu^{\beta}_{t})
 =\frac{\left(\hat Z^{\beta}_t , G \hat Z^{\beta}_t\right)}{\left(Z^{\beta}_t\right)^2}.
\end{equation}
Note that $(J_t)$ is bounded: since  $\mu^{\beta}_t$ is a probability measure we have, for any $x$,
$$G \mu^{\beta}_t(x)\le \max_{y\in \bbZ} G(x,y)= g(0)$$ and hence
\begin{equation}\label{boundJ}
0\le J_t\le g(0).
\end{equation}
Next, we consider the decomposition of $J_t$ into a martingale and a predictable process, which can be obtained by the use of Itô calculus.
\begin{proposition}\label{itoto}
We have
 \begin{equation}\label{itojj}
  J_t=g(0)+ N_t + A_t,
  \end{equation}
where
\begin{equation}\begin{split}\label{ntat}
N_t&\coloneqq 2\beta \int^t_0 \sum_{x\in \bbZ^d}\left( - J_s \mu^{\beta}_s(x)+\left(G \mu^{\beta}_{s}\right)(x)\mu^{\beta}_{s}(x)\right) \dd \eta^{(x)}_s,\\
A_t&\coloneqq \int^t_{0}\left[(\beta^2 g(0)-1)I_s+3\beta^2 I_sJ_s
  -4\beta^2\left((\mu^{\beta}_s)^2 , G\mu^{\beta}_s \right)\right]\dd s.
\end{split}\end{equation}

\end{proposition}

The proof of Proposition~\ref{itoto} is computation-intensive, so for expository reasons we postpone it to Section \ref{lastproof}.
We now give a quick overview of the proof of Proposition~\ref{keyprop2}.
The first important point, proved in Lemma~\ref{splash}, is that  the term $(\beta^2 g(0)-1)I_t$ dominates the two others terms appearing in the integrand of $A_t$ (namely, $I_sJ_s$ and $((\mu^{\beta}_t)^2 , G\mu^{\beta}_t )$) if $I_t$ is small.

\medskip

 Since $\beta^2 g(0)-1>0$ (cf. \eqref{g0positive}), on time intervals  $[s,t]$ where  $I_{\cdot}$ is small, the increment of $A_t-A_s$ is comparable to the  integral  $\int^t_{s}I_r\ \dd r$.
This implies that, if $K$ is chosen to be large, on the event
\begin{align*}
\left\{ \langle M \rangle_{\tau_{Ku}}-\langle M \rangle_{\tau_{u}}\ge\frac{\log K}{4}\ ; \ \max_{t\in(\tau_u,\tau_{Ku}]  } I_t\le \delta\right\}
\end{align*}
the increment $A_{\tau_{Ku}}-A_{\tau_u}$ is large. Since $J_t$ is bounded this must be compensated by having  $N_{\tau_{Ku}}-N_{\tau_u}$ be large in absolute value (and negative). We show that the probability of this latter event is small. This is done by computing the quadratic variation of $N$ and using Lemma~\ref{largedevcont}.

  \medskip
  
 The definition of $J_t$ may appear unmotivated at first. To get some intuition, notice that if we were to drop $G$ from the definition we are simply left with $I_t$, so $J_t$ measures how localized the polymer measure is. In fact, in \cite[Section 4]{CH06} they consider the process $I_t$ for the same purpose as $J_t$ defined above (note that the process $J_t$ appearing in \cite{CH06} has a different meaning). By including $G$ in \eqref{deffjt} we measure the localization of a ``smeared out'' version of $\mu_t^\beta$. A detailed explanation why this is desirable is given in Section~\ref{expl} after the conclusion of the calculations.

\smallskip At this point, we mention that the definition for $J_t$ is derived from a similar quantity defined in \cite[Section 3]{Y10} (for a discrete model, and using a truncated version of the partition function. See also \eqref{deffjn} in Section~\ref{discretos} of the present paper).
The Green operator $G$ seems a natural choice as a quantity to convolve with $\mu_t^\beta$ as it cancels with the Laplace operator $\Delta$ which appears naturally in the expression of $\dd \hat Z^{\beta}_{t}(x)$ -- this part yields a contribution $-I_t$ to $\dd A_t$.
The term $\beta^2 g(0) I_t$ corresponds to the quadratic variation terms coming from the numerator of \eqref{deffjt} (in the second expression involving $\hat Z^{\beta}_t$).
The other two terms come from the contribution of terms coming from the denominator, and from the interaction between numerator and denominator, respectively.

  \medskip

Let us now delve more into the proof of Proposition~\ref{keyprop2}. 
\begin{proof}[Proof of Proposition~\ref{keyprop2} assuming Proposition~\ref{itoto}]
For convenience, we introduce 
$ \tau'\coloneqq \inf \{t\geq \tau_u \ \colon  \ \int^t_{\tau_u} I_s \dd s=\frac{\log K}{4\beta^2} \}$.
Since on the event $\langle M\rangle_{\tau_{Ku}}-\langle M\rangle_{\tau_u}\geq \frac{\log K}4$ we have $\tau'\leq\tau_{Ku}$, it is sufficient to prove that 
\begin{equation}\label{mox}
   \bbP_u\left(  \max_{t\in(\tau_u,\tau']
   }  
I_t\le \delta\,;\,\tau'<\infty \right)\le \frac{1}{K^2}.
\end{equation}

We start with a technical result which allows to control the two last terms which appear in  \eqref{ntat} in terms of $I_t$.

\begin{lemma}\label{splash}
For every probability measure $\mu$ on $\bbZ^d$, it holds that, for all $x\in\Z^d$,
 \begin{equation}\label{3hold}
   \|G\mu\|_{\infty}\le \|g\|_4 \| \mu\|_2^{1/2}.
 \end{equation}
 where  $\|\cdot\|_{p}$ denote the $\ell_p$ norm for functions defined on $\bbZ^d$.
In particular,
 \begin{equation}\label{conseq}
  \|G \mu^{\beta}_t\|_{\infty} \le \|g \|_4 I^{1/4}_t, \quad J_t\le  \|g \|_4 I^{1/4}_t \quad \text{ and } \quad \left((\mu^{\beta}_t)^2 , G\mu^{\beta}_t \right)
  \le \|g\|_4 I^{5/4}_t.
 \end{equation}

\end{lemma}

\begin{remark}
 Here it is important to notice that since $g(x)\le C |x|^{2-d}$ in $d\ge 3$, we have $\|g\|_4<\infty$.
\end{remark}

\begin{proof}[Proof of Lemma~\ref{splash}]

The inequality is simply the multi-index H\"older inequality applied to the product the three functions with respective exponents $4$, $4$ and $2$,
\begin{equation}
 G\mu(x) =\sum_{y\in \bbZ^d} (g(x-y)\sqrt{\mu(y)}\sqrt{\mu(y)}
 \le \|g\|_4 \left( \sum_{y\in \bbZ^d} (\sqrt{\mu(y)})^4 \right)^{1/4}
 \left( \sum_{x\in \bbZ^d} (\sqrt{\mu(y)})^2 \right)^{1/2}.
\end{equation}
The inequalities in \eqref{conseq} immediately follow by applying \eqref{3hold} to $\mu^{\beta}_t$.
\end{proof}

% 
%  \begin{lemma}\label{boundonA}
% If $\tau_u<\infty$ and  $\max_{t\in (\tau_u,\tau'_u]}I_t< \delta$ then for $\delta$ sufficiently small we have
% \begin{equation}
% A_{\tau'_u}-A_{\tau_{u}}\ge \frac{\left(\beta^2 g(0)-1\right)(\log K) }{5}.
% \end{equation}
% 
%  \end{lemma}

Now, as a consequence of Lemma~\ref{splash} and \eqref{ntat}, we have  
\begin{equation}
A_{\tau'}-A_{\tau_{u}}\ge \int^{\tau'}_{\tau_u} \left[(\beta^2 g(0)-1)I_s
  -4\beta^2\|g\|_{4} I^{5/4}_s\right]\dd s.
\end{equation}
Hence, if $\tau'<\infty$ and $\max_{t\in(\tau_u,\tau']  }   I_t\le \delta$, we have 
  \begin{equation}
    A_{\tau'}-A_{\tau_{u}}\ge \int^{\tau'}_{\tau_u} (\beta^2 g(0)-1-4\beta^2\|g\|_4 \delta^{1/4})I_s\dd s =\frac{(\beta^2 g(0)-1-4\beta^2 \|g\|_4 \delta^{1/4})\log K}{4 \beta^2}.
  \end{equation}
This implies that 
\begin{equation}
   N_{\tau'}-N_{\tau_{u}}=J_{\tau'}-J_{\tau_u}- A_{\tau'}+A_{\tau_{u}}\le g(0)-\frac{(\beta^2 g(0)-1-4\beta^2 \|g\|_4  \delta^{1/4})\log K}{4 \beta^2}.
\end{equation}
and hence if $\delta$ is sufficiently small and $K$ is sufficiently large
\begin{equation}\label{bbb}
N_{\tau'}-N_{\tau_{u}} \le  - \frac{\left(\beta^2 g(0)-1\right)(\log K) }{5 \beta^2},
\end{equation}
where we recall that the numerator is positive due to \eqref{g0positive}.
On the other hand, we have 
\begin{equation}
	\langle N \rangle_{t}= 4\beta^2\int^{t}_{0}\sum_{x\in \bbZ^d}\mu^{\beta}_{t}(x)^2\left( G \mu^{\beta}_{t}(x)- J_t\right)^2\dd t.
\end{equation}
If   $\max_{t\in (\tau_u,\tau']}I_t< \delta$ then we obtain
 \begin{equation}\label{ccc}
  \langle N\rangle_{\tau'} -\langle N\rangle_{\tau_{u}}
  \le 4\beta^2 \| g\|^2_4 \delta^{1/2} \int^{\tau'}_{\tau_u} I_s \dd s=
  \beta^2 \| g\|^2_4 \delta^{1/2} \log K.
 \end{equation}
 We have used Lemma~\ref{splash} to bound $J_t$ and  $G \mu^{\beta}_{t}(x)$ and applied the inequality $(a-b)^2\leq \max(a^2,b^2)$ (valid for $a,b\geq 0$).
 We are finally ready to prove \eqref{mox}. Indeed from \eqref{bbb}--\eqref{ccc} (which both hold when $\tau'<\infty$ and $\max_{t\in (\tau_u,\tau']}I_t< \delta$) we have 
 \begin{multline}
   \bbP_u\left(  \max_{t\in(\tau_u,\tau']
   }  
  I_t\le \delta\;;\;\tau'<\infty \right)\\ \le 
  \bbP_u\left(  N_{\tau'}-N_{\tau_{u}} \le  - \frac{\left(\beta^2 g(0)-1\right)\log K }{5\beta^2}  \ ; \     \langle N\rangle_{\tau'} -\langle N\rangle_{\tau_{u}}\le   \beta^2 \| g\|^2_4 \delta^{1/2} \log K \right).
 \end{multline}
Applying Lemma~\ref{largedevcont} to the martingale $-(N_{t+\tau_u}-N_{\tau_u})$ (with $v=  \frac{\left(\beta^2 g(0)-1\right)\log K }{5\beta^2}$, and $a=\beta^2 \| g\|^2_4 \delta^{1/2} \log K$ -- we again note that the proof of the lemma works \textit{verbatim} with conditional expectation) we obtain that
\begin{equation}   \bbP_u\left(  \max_{t\in(\tau_u,\tau']
   }  
  I_t\le \delta,\tau'<\infty \right)
 \le \exp\left( - \frac{\left(\beta^2 g(0)-1\right)^2\log K}{50 \beta^6 \|g\|^2_4\delta^{1/2}}\right)\le K^{-2}
\end{equation}
where the last inequality is satisfied for $\delta$ sufficiently small. 
\end{proof}

\subsection{Proof of Proposition~\ref{itoto}}\label{lastproof} In
	 order to make the computation easier to follow, we start by decomposing the semimartingale $\left(\hat Z^{\beta}_t , G \hat Z^{\beta}_t\right)$ into a martingale part and a finite variation part. 
We start by observing that for each $x\in \bbZ^d$
\begin{equation}\label{dhatz}
 \dd \hat Z^{\beta}_t(x)= \beta  \hat Z^{\beta}_t(x)\dd \eta^{(x)}_t + \left(\Delta  \hat Z^{\beta}_t\right)(x)\dd t,
\end{equation}
that is to say $\hat Z^{\beta}_{\cdot}(\cdot)$ is the solution of the discrete-space continuous-time stochastic Heat equation with multiplicative noise (we refer to e.g. \cite[Proposition 5]{CH06} for a proof of this fact).
Hence we have, recalling that $G=-(2\Delta)^{-1}$,
\begin{equation}\begin{split}\label{dghatz}
 \dd \left(G \hat Z^{\beta}_t\right)(x)&=\beta\sum_{y\in \bbZ^d} g(x-y)\hat Z^{\beta}_{t}(y) \dd  \eta^{(y)}_t - \frac{1}{2}\hat Z^{\beta}_t(x)\dd t,\\
 \dd Z^{\beta}_t&= \beta \sum_{x\in \bbZ^d} \hat Z^{\beta}_t(x)\dd \eta^{(x)}_t.
 \end{split}
\end{equation}
\begin{remark}
The two lines of \eqref{dghatz} are obtained by taking the linear combination of countably many stochastic differential equations and similar operations are performed in the remainder of this proof starting with \eqref{dzgz}.
To justify this, we consider the truncated sums and check that both sides convergence in $L^2$. In order to keep this intermezzo section short we leave this verification, which presents no difficulty, to the reader.
\end{remark}
Now we can use Itô's formula for the scalar product and obtain
\begin{equation}\label{dzgz}
\dd \left(\hat Z^{\beta}_t , G \hat Z^{\beta}_t\right)
= 2\left( \hat Z^{\beta}_t , \dd G \hat Z^{\beta}_t\right)
+\sum_{x\in \bbZ^d } \dd \langle \hat Z^{\beta}(x), G \hat Z^{\beta}(x)\rangle_t.
\end{equation}
In the above formula we have used that $( \dd \hat Z^{\beta}_t , G \hat Z^{\beta}_t)
 =(  \hat Z^{\beta}_t ,  \dd G \hat Z^{\beta}_t)$ by self-adjointness of $G$.
We can compute the first term in the r.h.s.\ of \eqref{dzgz} using  \eqref{dhatz} and \eqref{dghatz}
\begin{equation}
 \left(  \hat Z^{\beta}_t ,  \dd G \hat Z^{\beta}_t\right)
 =\beta \sum_{x\in \bbZ^d} \hat Z^{\beta}_t(x) G \hat Z^{\beta}_t(x)
 \dd  \eta^{(x)}_t - \frac{1}{2}  \left( \hat Z^{\beta}_t ,  \hat Z^{\beta}_t\right)\dd t.
 \end{equation}
 More precisely, we use \eqref{dhatz} to compute the martingale part of  $(\dd   \hat Z^{\beta}_t ,  G \hat Z^{\beta}_t)$
(which is the same as that of  $(  \hat Z^{\beta}_t , \dd  G \hat Z^{\beta}_t)$) and \eqref{dghatz} for the predictable part. 
 Also using also \eqref{dhatz}-\eqref{dghatz}, we obtain
\begin{equation}
	\sum_{x\in\Z^d} \dd \langle \hat Z^{\beta}(x), G \hat Z^{\beta}(x)\rangle_t
 =\beta^2 g(0) \left( \hat Z^{\beta}_t ,  \hat Z^{\beta}_t\right)\dd t=\beta^2g(0) (Z^{\beta}_t)^2   I_t \dd t
\end{equation}
and altogether
\begin{equation}\label{itodenominator}
 \frac{\dd \left(\hat Z^{\beta}_t , G \hat Z^{\beta}_t\right)}{(Z^{\beta}_t)^2}= 2\beta \sum_{x\in \bbZ^d} \mu^{\beta}_t(x) G \mu^{\beta}_t(x)
 \dd  \eta^{(x)}_t+ (\beta^2 g(0)-1)I_t \dd t.
\end{equation}
To compute the infinitesimal variation of $J_t$, we apply Itô's formula to $f\big((\hat Z^{\beta}_t , G \hat Z^{\beta}_t)\ , Z^{\beta}_{t}\big)$ where $f(a,b)=a/b^2$.
 We obtain
 \begin{equation}\label{itoj}
  \dd J_t=\frac{\dd \big(\hat Z^{\beta}_t , G \hat Z^{\beta}_t\big)}{(Z^{\beta}_t)^2}-\frac{2 \dd Z^{\beta}_t \big(\hat Z^{\beta}_t , G \hat Z^{\beta}_t\big) }{ (Z^{\beta}_t)^3}+3\frac{ \dd \langle Z^{\beta}\rangle_t \big(\hat Z^{\beta}_t , G \hat Z^{\beta}_t\big)  }{ (Z^{\beta}_t)^4}
  -2 \frac{\dd \big\langle Z^{\beta}_t ,\big(\hat Z^{\beta} , G \hat Z^{\beta}\big)  \big\rangle_t}{(Z^{\beta}_t)^3}.
 \end{equation}
 We compute each term separately. 
Using \eqref{dghatz} (the second line) we have
\begin{equation}
 \dd \langle Z^{\beta}\rangle_t= \beta^2 \big( \hat Z^{\beta}_t ,  \hat Z^{\beta}_t\big)\dd t= \beta^2 (Z^{\beta}_t)^2   I_t \dd t.
\end{equation}
Combining \eqref{dghatz} with \eqref{dzgz} we have
 \begin{equation}
  \dd \big\langle Z^{\beta}_t ,\big(\hat Z^{\beta} , G \hat Z^{\beta}\big)  \big\rangle_t= 2\sum_{x\in \bbZ^d} \hat Z^{\beta}_t(x)^2 G\hat Z^{\beta}_{t}(x) \dd t=2\big ( (\hat Z^{\beta}_t)^2 , G\hat Z^{\beta}_{t}\big)\dd t.
 \end{equation}
 Replacing all terms in \eqref{itoj} by the expressions we have computed, and using the fact that $\hat Z^{\beta}_t/Z^{\beta}_t=\mu^{\beta}_t$ we obtain
 \begin{multline}
  \dd J_t= 2\beta\sum_{x\in \bbZ^d}\left(\mu^{\beta}_t(x) G \mu^{\beta}_t(x)-\mu^{\beta}_t(x) J_t\right)\dd \eta^{(x)}_t \\
  + \left[(\beta^2 g(0)-1)I_t+3\beta^2 I_tJ_t
  -4\beta^2\left ( (\mu^{\beta}_t)^2 , G\mu^{\beta}_t \right)\right]\dd t.
 \end{multline}
which is the required result.
 \qed

\subsection{Additional insight concerning the choice of $J_t$}\label{expl}
Let us now take advantage of the computation made in the previous section to provide further justification for the choice of $J_t$.
Given a function $f:\bbZ^d\to \bbR_+$ such that $\|f\|_4<\infty$ ( this restriction is taken only in view of Lemma \ref{splash}), let us define  $$J^{(f)}_t:=(\mu^{\beta}_t, f\ast \mu^{\beta}_t)= (\hat Z^{\beta}_t, f\ast \hat Z^{\beta}_t) / (Z^{\beta}_t)^2$$ where $\ast$ denotes the convolution. 
Repeating the computation leading to \eqref{itodenominator}, the reader might check that 
\begin{equation}\label{ddroipio}
 \frac{\dd (\hat Z^{\beta}_t, f\ast \hat Z^{\beta}_t) }{(Z^{\beta}_t)^2}= 2\beta \sum_{x\in \bbZ^d} \mu^{\beta}_t(x) (f\ast \mu^{\beta}_t)(x) \dd \eta_t(x) + \left(\mu^{\beta}_t, f'\ast \mu^{\beta}_t \right)\dd t
\end{equation}
where the function $f'$ is defined by  $f'(x):= 2 (\Delta f)(x)+ \beta^2 f(0)$, or to put is more simply $f'= (2\Delta +\beta^2\delta_0)f$.
The crucial point that makes the proof in Section \ref{lastproof} work is that we have chosen $f$ such that $f'=c_{\beta}\delta_0$ for some $c_{\beta}\ne 0$ so that $c_\beta I_t$ appears on the r.h.s.\ of \eqref{ddroipio}. The other terms that appear in the drift part of $\dd J^{(f)}_t$ involve third or fourth power of $\mu^{\beta}_t$ and are thus negligible w.r.t.\ to $I_t$ when $I_t<\delta$ (this is rigorously taken care of by Lemma \ref{splash}).

\medskip

The value $\beta_2$ it a threshold value. The operator  $(2\Delta +\beta^2\delta_0)$ has a positive eigenvalue if and only if $\beta>\beta_2$.
The function $g$ defined in \eqref{defg} corresponds to the eigenfunction  of the operator $(2\Delta +\beta^2_2\delta_0)$ associated with the eigenvalue $0$ and hence hence $g'= (\beta^2-\beta^2_2)g(0)\delta_0$ (strictly speaking $g\in \ell_2$ only if $d\ge 5$ but it is a minor point for the sake of this discussion).

\section{Proof of Theorem~\ref{locaend}}\label{discretos}

\subsection{Organization of the proof}

In this Section we adapt the proof presented in the previous section to the discrete setup of directed polymer.
In analogy with $(M_t)$ we set for $n\ge 1$,  $M_{n}\coloneqq \sum_{i=1}^n\left(\frac{W^{\beta}_{i+1}}{W^{\beta}_i}-1\right).$
The process $(M_n)$ is a martingale and the convexity of $\log$ implies that
\begin{equation}\label{telescop}
\log W^{\beta}_n=\sum_{i=1}^n \log \left(\frac{W^{\beta}_{i+1}}{W^{\beta}_i}\right)\le M_n.
\end{equation}
Furthermore, as $\go_{1,0}$ is bounded from above, $M_n$ has bounded increments, namely
$ M_{n+1}-M_{n}\in[-1, L-1]$, where $L$ is as in \eqref{fromabove}. 
Using \eqref{condvar} we obtain that the predictable bracket of the martingale $M_n$ is given by
\begin{equation}\label{disbracket}
\langle M\rangle_n\coloneqq  \sum_{k=1}^n \bbE\left[ (M_{n}-M_k-1)^2 \ | \ \cF_{k-1} \right]=  \chi(\beta) \sum_{k=1}^n I_{k}.
\end{equation}
We set, in analogy with \eqref{tauucont}, given $u\ge 1$ and an event $A$,
\begin{equation}
\tau_u\coloneqq \inf\{ n\ge 0 \ \colon  \ W^{\beta}_n\ge u\} \quad \text{ and } \quad \bbP_u[A]\coloneqq  \bbE\left[\ind_A \middle| \cF_{\tau_u}\right].
\end{equation}
The first step in the proof is to show (as in Proposition~\ref{keyprop1}) that with large probability the increment of the bracket of $M$ over the interval $(\tau_u,\tau_{Ku}]$ is large.

\begin{proposition}\label{keyprop1dis}
If strong disorder holds then there exists  a constant $\theta>0$ (which depends only on $L$) such that for every $u$ and $K\ge L^2$, on the event $\tau_u<\infty$,
 \begin{equation}\label{disversion}
  \bbP_u\left( \langle M\rangle_{\tau_{Ku}}-\langle M\rangle_{\tau_{u}}\le \theta \log K    \ ; \ \tau_{Ku}<\infty \right)\le K^{-2}.
 \end{equation}

\end{proposition}

The second step (the analogue of  Proposition~\ref{keyprop1}) establishes that with large probability, if $\langle M\rangle_{\tau_{Ku}}-\langle M\rangle_{\tau_{u}}$ is large then $I_n$ has to be large somewhere in the interval $(\tau_u,\tau_{Ku}]$. 

\begin{proposition}\label{keyprop2dis}
 If strong disorder holds, then for any $\theta>0$ there exists a constant $\delta>0$ such that, for all $K$ large enough,  on the event $\tau_u<\infty$,
 \begin{equation}
  \bbP_u\left(   \langle M \rangle_{\tau_{Ku}} -\langle M \rangle_{\tau_u}\ge \theta \log K \ ; \  \max_{n\in(\tau_u,\tau_{Ku}]}
  I_n\le \delta \right)\le K^{-2}.
 \end{equation}

\end{proposition}

Finally, in a third step (analogous to Lemma~\ref{keylemma}), we guarantee that with large probability, there are no significant dips 
of $W^{\beta}_n$ between $\tau_u$ and $\tau_{Ku}$.
We set $
 \sigma_{u,K}\coloneqq  \min\{ n \ge \tau_u \ \colon  \ W^{\beta}_n\le \frac{u}{K}\}$.

\begin{lemma}\label{keylemmadis}
On the event $\tau_u<\infty$, we have
\begin{equation}
 \bbP_u( \sigma_{u,K}<\tau_{Ku}<\infty)\le K^{-2}.
\end{equation}
\end{lemma}
\begin{proof}
 Using the optional stopping theorem for the martingale 
 $M_{\sigma_{u,K}+s}$
 we obtain that on the event $\{\tau_u<\infty\}\cap\{ \sigma_{u,K}<\tau_{Ku}\}$ we have 
 \begin{equation}
  \bbP\left[ \tau_{Ku}<\infty | \ \cF_{\sigma_{u,K}} \right]
  \le \frac{1}{K^2},
 \end{equation}
which implies the desired result after integration.
\end{proof}

Like in the continuous setup, the hard work is the proof of Proposition~\ref{keyprop2dis}.
We postpone it to the next section. The proof of Proposition~\ref{keyprop1dis} relies on a concentration property for the martingale which is analogous to Lemma~\ref{largedevcont} (recall the definition of the martingale bracket in discrete time \eqref{disbracket}).

\begin{lemma}\label{largedevdis}
 Let $(N_n)$ be a discrete time martingale starting at $0$, with increments that are bounded in absolute value by $A> 0$. For $v>0$, let $T_v$ be the first time that
 $(N_n)$ hits $[v,\infty)$. For any $\alpha>0$, we have
 \begin{equation}
	 \bbP[ \langle N\rangle_{T_v}\le  a \ ; \ T_v<\infty ] \le e^{-\frac{v}{A+1}(\log \frac{v}{a(A+1)}-1)}.% e^{v(1-\frac{\log(v/a)}{A+1}) }.
 \end{equation}
\end{lemma}

\begin{proof}

Given $\gl>0$, we have for every $u\in [-A,A]$
\begin{equation*}
(e^{\gl u} - \gl u -1)\le  \frac{e^{A\gl}}{2} \left(\gl u\right)^2\le
e^{(A+1)\gl}u^2.
\end{equation*}
Hence we have
\begin{multline}
\bbE[ e^{\gl(N_{n+1}-N_n)} \ | \ \cF_n]
= \bbE[ e^{\gl(N_{n+1}-N_n)}-\gl(N_{n+1}-N_n)-1 \ | \ \cF_n]\\ \le  1+ e^{(A+1)\gl}
\bbE\left[ (N_{n+1}-N_n)^2 \ | \ \cF_n \right] 
\le e^{ e^{(A+1)\gl} \bbE\left[ (N_{n+1}-N_n)^2 \ | \ \cF_n \right]}.
\end{multline}
This implies that the process
 $(e^{ \gl  N_n - e^{(A+1)\gl}\langle N\rangle_n})_{n\ge 0}$
is a supermartingale.
Hence using the optional stopping theorem we have
\begin{equation}
\bbP\left[  \langle N\rangle_{T_v}\le  a \ ; \ T_v<\infty  \right]\le
  e^{-(\gl v - e^{(A+1)\gl}  a) }.
\end{equation}
We obtain the desired result by taking $\gl=\frac{1}{A+1} \log(\frac{v}{a(A+1)})$.
\end{proof}

\begin{proof}[Proof of Proposition~\ref{keyprop1dis}]
 We apply Lemma~\ref{largedevdis} to the martingale $(M_{\tau_u+s}-M_{\tau_u})_{s\ge 0}$ with $A=\max(L-1,1)$, $a=\theta\log K$ and $v=\log \sqrt{K}$.
 We have 
 \begin{equation}\label{petiteborne}
 \bbP_u\left[ \langle M\rangle_{\tau_u+T_v}-\langle M\rangle_{\tau_{u}}\le \theta \log K \ ; \ T_v<\infty \right] 
 \le K^{-\frac{1}{2(A+1)}(\log\frac{1}{2\theta(A+1)}-1)}\le K^{-2},
\end{equation}
where the last inequality holds for $\theta$ small enough.
% where the latter inequality is achieved by taking $\theta=1/[10(A+1)]$.
Note that the boundedness of the environment implies $M_{\tau_u}\leq Lu$, so from \eqref{step1} and the telescopic argument in \eqref{telescop} we have
\begin{equation}
\log \sqrt{K}\le \log \left(\frac{K}{L}\right)\le  \log \frac{W^{\beta}_{\tau_{Ku}}}{W^{\beta}_{\tau_u}}\le M_{\tau_{Ku}}-M_{\tau_u},
\end{equation}
which guarantees that $T_v\le \tau_{Ku}-\tau_u$. Since $n\mapsto\langle M\rangle_n$ is increasing, \eqref{petiteborne} implies \eqref{disversion}.
\end{proof}

 Let us now briefly explained how the combination of the above yields our result repeating the argument displayed in the proof of Theorem~\ref{locaend2}.
\begin{proof}[Proof of Theorem~\ref{locaend}]
Similarly to \eqref{incluclu} we have
\begin{equation*} \left\{  \tau_{Ku}<\sigma_{u,K} \ ; \ \max_{n\in(\tau_u,\tau_{Ku}]  }
  I_n> \delta  \right\}\subset\left\{ \maxtwo{n\ge 0}{x\in \bbZ^d} \hat W^{\beta}_n(x)\ge \frac{\delta u}{K} \right\}.
  \end{equation*}
  Indeed if $I_n\ge \delta$ for some $n\in(\tau_u,\tau_{Ku}]$ and  $\tau_{Ku}<\sigma_{u,K}$ then we have $W^{\beta}_{n-1}\ge u/K$ and thus
  \begin{equation}
   \max_{x\in \bbZ^d} \hat W^{\beta}_{n-1}(x)\ge    \max_{x\in \bbZ^d} D\hat W^{\beta}_{n-1}(x)= W^{\beta}_{n-1} \max_{x\in \bbZ^d} \mu^{\beta}_{n-1}(x)\ge W^{\beta}_{n-1} I_n\ge \frac{\delta u}{K}.
  \end{equation}
Now, using Lemma~\ref{dull}, Propositions~\ref{keyprop1dis} and~\ref{keyprop2dis} as well as Lemma~\ref{keylemmadis}, we can repeat the computation~\ref{frombelow} and deduce from the above inclusion that

\begin{equation}
 \bbP\left( \maxtwo{n\ge 0}{x\in \bbZ^d} \hat W^{\beta}_n(x)\ge \frac{\delta u}{K}\right)\ge \frac{1}{u}\left(\frac{1}{LK}-\frac{3}{K^2}\right).
\end{equation}
which yields the desired result provided that we chose $K\geq 6L$.
\end{proof}

\subsection{Proof of Proposition~\ref{keyprop2dis}}

The proof is in essence the same as in the continuous case but it is slightly more technical since we cannot resort to the help of stochastic calculus. The idea remains to extract the result from the study of a process $(J_n)_{n\ge 0}$ which is adapted to the filtration $(\cF_n)$ and bounded. We use the Doob decomposition $J_n=A_n+N_n$, where $A_n$ is a predictable process and $N_n$ is a martingale.
We want to show that whenever $I_n$ is small, then $A_n-A_{n-1}$ is comparable to $I_n$, meaning that if $\max_{n\in (\tau_u,\tau_{Ku}]}I_n$ is small then $A_{\tau_{Ku}}-A_{\tau_{u}}$ is comparable to $\sum_{n=\tau_{u}+1}^{\tau_{Ku}} I_n$ which is large by Proposition~\ref{keyprop1dis}. Since $J_n$ is bounded, this has to be compensated by the martingale part increment $N_{\tau_{Ku}}-N_{\tau_u}$ being very negative. We show that the probability of the latter event is small using Lemma~\ref{largedevdis}.

\medskip

To get started, let us define the process $J_n$, which as in the continuous case, is defined in terms of the Green's function and the endpoint distribution. For a technical reason (which we expose after Lemma~\ref{aincrement}) we consider a truncated version of the Green function.
Using Theorem~\ref{b2bcrit} and \eqref{defbeta2}, since strong disorder holds, we can find $n_0(\beta)$ which is such that 
\begin{equation}
  \chi(\beta)\sum_{n=1}^{n_0}P^{\otimes 2} (X^{(1)}_n=X^{(2)}_{n})>1.
\end{equation}
With such an $n_0$ being fixed, we introduce $G_0$, the  Green function truncated at level $n_0$
\begin{equation}\begin{split}\label{defg0}
 g_{0}(x)&\coloneqq  \sum_{n=1}^{n_0}P^{\otimes 2} (X^{(1)}_n-X^{(2)}_{n}=x)= \sum_{n=1}^{n_0}P(X_{2n}=x),\\
 G_{0}(x,y)&\coloneqq g_{0}(y-x).
\end{split}\end{equation}
Recalling that $D$ is the transition matrix of the simple random walk, we have
 $G_0=\frac{D^2-D^{2(n_0+1)}}{1-D^2}$.
Next we define
\begin{equation}\label{deffjn}
 J_n\coloneqq  \left( \mu^{\beta}_{n} , G_0 \mu^{\beta}_{n}\right)=
 \frac{\left( \hat W^{\beta}_{n} , G_0 \hat W^{\beta}_{n}\right)}{ (W^{\beta}_n)^2}.
 \end{equation}
 We consider Doob's decomposition of $J_n$
 \begin{equation}
  J_n-J_0= N_n+ A_n,
 \end{equation}
where $A_n$ and $N_n$ are defined by $A_0=0$, $N_0=0$ and
\begin{equation}\begin{split}
A_{n+1}&=A_n+ \bbE\left[ J_{n+1}-J_n \ | \ \cF_n\right], \\
N_{n+1}&=N_n+ (J_{n+1}-J_n)-\bbE\left[ J_{n+1}-J_n \ | \ \cF_n\right].
\end{split}
\end{equation}
Without the help of stochastic calculus, we cannot derive ``simple'' expressions for $A_n$ and $N_n$ as in Proposition~\ref{itoto}.
However, building on similar ideas we can still use similar ideas to derive a useful bound on $A_n$.

\begin{lemma}\label{aincrement}
 We have
 \begin{equation*}
A_{n}-A_{n-1} \ge  \left(\chi(\beta) g_0(0)-1\right)I_n-4\chi(\beta)\left(\left(D \mu^{\beta}_{n-1}\right)^2, G_0\mu^{\beta}_{n-1}\right)
-2\chi_3(\beta) \sum_{x\in \bbZ^d }D \mu^{\beta}_{n-1}(x)^3.
 \end{equation*}
where $\chi_3(\beta)\coloneqq \bbE\left[ (e^{\beta \go_{1,0}-\gl(\beta)}-1)^3\right]$.

\end{lemma}

In order to control the predictive bracket $\langle N\rangle_n$, instead of an explicit computation of $N_n$  (which would be quite tedious), we rely on a simple bound for $J_n$, and on the observation that $|N_n-N_{n-1}|\le |J_n-J_{n-1}|$.
 From Lemma \ref{splash} we have $0 \le J_n\le \|g_0\|_4 \|\mu^{\beta}_n \|^{1/2}_2$. The increments of the martingale $(N_n)_{n\ge 0}$ can thus be bounded in the following manner
\begin{equation}\label{tprop}
 |N_n-N_{n-1}|\le \|g_0\|_4 \max\left( \|\mu^{\beta}_n\|^{1/2}_{2} ,\|\mu^{\beta}_{n-1}\|^{1/2}_{2} \right).
\end{equation}
 The above observation remains valid even without truncating the Green function but unfortunately is not good enough for our purpose. Using the Cauchy-Schwarz inequality (first) and then the inequality
$\|f\ast g\|_2\le \|f\|_1 \|g\|_2$, we obtain that
\begin{equation}\label{borne}
 0\le J_n\le \|G_0 \mu^{\beta}_{n} \|_2 \| \mu^{\beta}_{n} \|_2\le \|g_0\|_1 \|\mu^{\beta}_{n} \|^2_2= n_0  \|\mu^{\beta}_{n} \|^2_2.
\end{equation}
We obtain thus a variant of \eqref{tprop} with an improved exponent for $\|\mu^{\beta}_{n}\| _2$ and $\|\mu^{\beta}_{n-1}\|_2$, that is
\begin{equation}\label{tprop2}
  |N_n-N_{n-1}|\le n_0 \max\left( \|\mu^{\beta}_{n} \|^2_2, \|\mu^{\beta}_{n} \|^2_2\right).
\end{equation}
This change of exponent has a crucial importance in \eqref{bquadra}.
We use it to prove the following.
\begin{lemma}\label{Nquadratic}
We have for every $n$
\begin{equation}\label{ptrop3}
\bbE\left[ (N_{n}-N_{n-1})^2 \ | \ \cF_{n-1}\right]\le \kappa I^2_n,
\end{equation}
where $ \kappa\coloneqq n^2_0 \left[  (2d)^2+ e^{\frac{\gl(8\gb)+\gl(-8\gb)}{2}}\right]$.
\end{lemma}
We postpone the proof of Lemmas~\ref{aincrement} and~\ref{Nquadratic} to the end of the section and 
show first how Proposition~\ref{keyprop2dis} may be deduced from them.

\begin{proof}[Proof of Proposition~\ref{keyprop2dis} assuming Lemmas~\ref{aincrement} and~\ref{Nquadratic}]
	Given $\theta>0$,  let us set 
$$\tau'\coloneqq \inf_{n\ge 0}\left\{ \langle M\rangle_n- \langle M\rangle_{\tau_u} \ge \theta \log K  \right\}.$$
Since on the event  $\langle M\rangle_{\tau_{Ku}}- \langle M\rangle_{\tau_u}\ge \theta \log K$ we have $\tau'\le \tau_{Ku}$, it is sufficient for our purpose to show that 
\begin{equation}\label{mox2}
 \bbP\left[ \max_{n\in (\tau_u,\tau']} I_n\le \delta,\tau'<\infty \right] \le K^{-2}.
\end{equation}
The first step is to show that $A_{\tau'}-A_{\tau_u}$ is large on the event above. To bound  the second term appearing in the r.h.s.\  of the inequality in Lemma~\ref{aincrement}, we use Lemma~\ref{splash} (with $G$ replaced by $G_0$) to obtain
\begin{equation}
 \left(\left(D \mu^{\beta}_{n-1}\right)^2, G_0\mu^{\beta}_{n-1}\right)\le \|g_0\|_4I^{5/4}_n \text{ and }
 \sum_{x\in \bbZ^d }D \mu^{\beta}_{n-1}(x)^3\le  \| D \mu^{\beta}_{n-1} \|^{3}_2= I^{3/2}_n.
\end{equation}
In the second inequality  we used the fact $\|\cdot \|_{p}\le \|\cdot\|_q$ when $p\ge q$.
Hence if $I_n\le \delta$ we have 
\begin{equation}
 A_n-A_{n-1}\ge I_n \left(\chi(\beta) g_0(0)-1- 4\delta^{1/4}\|g_0\|_4 \chi(\beta)-2\delta^{1/2} \chi_3(\beta)\right),
\end{equation}
and thus if $\tau'<\infty$ and $\max_{n\in (\tau_u,\tau']} I_n\le \delta$ we have (recall \eqref{disbracket})
\begin{equation}\label{controla}\begin{split}
 A_{\tau'}-A_{\tau_u}&\ge \left(\sum_{n=\tau_u+1}^{\tau'} I_n \right)\left(\chi(\beta) g_0(0)-1- 4\delta^{1/4}\|g_0\|_4 \chi(\beta)-2\delta^{1/2} \chi_3(\beta)\right)\\
 &= \frac{\left(\langle M\rangle_{\tau'}-\langle M\rangle_{\tau_u}\right) \left(\chi(\beta) g_0(0)-1- 4\delta^{1/4}\|g_0\|_4 \chi(\beta)-2\delta^{1/2} \chi_3(\beta)\right) }{\chi(\beta)}\\
 &\ge \frac{\theta \log K \left(\chi(\beta) g_0(0)-1- 4\delta^{1/4}\|g_0\|_4\chi(\beta)-2\delta^{1/2} \chi_3(\beta)\right) }{\chi(\beta)}.
 \end{split}
\end{equation}
Using \eqref{borne} and a trivial bound for $\|\mu^{\beta}_n\|_2$, we have
\begin{equation}\label{controln}
 N_{\tau'}-N_{\tau_u}=A_{\tau_u}-A_{\tau'}+J_{\tau'}-J_{\tau_u}
 \le A_{\tau_u}-A_{\tau'}+n_0.
\end{equation}
Choosing $\delta$ sufficiently small, \eqref{controln} and \eqref{controla} imply that, for all $K$ sufficiently large,
\begin{equation}\label{bn}
  N_{\tau'}-N_{\tau_u}\le - \frac{\theta\log K(\chi(\beta) g_0(0)-1)}{2 \chi(\beta)}.
\end{equation}
On the other hand, on the event $\max_{n\in (\tau_u,\tau']} I_n\le \delta$, we have from Lemma~\ref{Nquadratic} and \eqref{disbracket},
\begin{equation}\label{bquadra}
 \langle N\rangle_{\tau'}-\langle N\rangle_{\tau_u}\le \kappa \sum_{n=\tau_u+1}^{\tau'} I^2_n\le  \kappa \left(\delta \sum_{n=\tau_u+1}^{\tau'-1} I_n+\delta^2\right) \le \kappa\delta \left( \frac{\theta \log K}{\chi(\beta)} +\delta \right).
\end{equation}
Thus we obtain
\begin{multline}
  \bbP_u\left(  \max_{n\in (\tau_u,\tau']}
  I_n\le \delta\ ; \ \tau'<\infty \right) \\  \le \bbP_u\left
( N_{\tau'}-N_{\tau_u}\le - \frac{\theta\log K(\chi(\beta) g_0(0)-1)}{2 \chi(\beta)} \ ; \    \langle N\rangle_{\tau'}-\langle N\rangle_{\tau_u}  \le\kappa\delta \left( \frac{\theta \log K}{\chi(\beta)} +\delta \right) \right).
\end{multline}
Using Lemma~\ref{largedevdis} for the martingale $(N_{\tau_u}-N_{\tau_u+m})_{m\ge 0}$ with $A=n_0$ (which is a bound for the increments of $N$ by  \eqref{borne}), $v=\frac{\theta\log K(\chi(\beta) g_0(0)-1)}{2 \chi(\beta)}$ and  $a=\kappa \delta\left(  \frac{\theta \log K}{\chi(\beta)} +\delta \right)$ since $T_v\le \tau'-\tau_u$ on the above event we have 
\begin{equation}
  \bbP_u\left(  \max_{n\in (\tau_u,\tau']}
  I_n\le \delta \ ; \ \tau'<\infty\right)\le \bbP_u\left(T_v<\infty  \ ; 
  \ \langle N\rangle_{T_v+\tau_u}-\langle N\rangle_{\tau_u}\le a \right) \le e^{-\frac{v}{n_0+1}(\log(\frac{v}{a(n_0+1)})-1)}.
\end{equation}
To conclude the proof of \eqref{mox2} and hence of Proposition~\ref{keyprop2dis}, we need to check that the exponent is smaller than $-2\log K$, 
which is immediate if one chooses $\delta$  sufficiently small.
\end{proof}

\begin{proof}[Proof of Lemma~\ref{aincrement}]
We want to split $A_{n}-A_{n-1}$ into two parts: one corresponding to the addition of the $n$-th random walk step, and the other to reveal the environment $(\go_{n,x})_{x\in \bbZ^d}$.
Recalling that $D$ denotes the transition matrix of the random walk we have
\begin{equation}
\hat W^{\beta}_{n}(x)=  e^{\beta\go_{n,x}-\gl(\beta)}(D \hat W^{\beta}_{n-1})(x).
\end{equation}
and thus in particular 
\begin{equation}\label{martindecompo}
 \bbE[\hat W^{\beta}_{n} \ | \ \cF_{n-1}]=  D \hat W^{\beta}_{n-1}.
\end{equation}
We set $
 \tilde J_n\coloneqq  \frac{\left(G_0 D\hat W^{\beta}_{n-1} ,  D\hat W^{\beta}_{n-1}\right)}{ (W^{\beta}_{n-1})^2}$
and consider the decomposition
\begin{equation}\label{splitintwo}
 A_{n}-A_{n-1}=\bbE\left[ J_n-J_{n-1} \ | \ \cF_{n-1}\right]=\bbE\left[ J_n-\tilde J_{n} \ | \ \cF_{n-1}\right]+ (\tilde J_n-J_{n-1}).
\end{equation}
Using self-adjointness and the fact that $G_0$ and $D$ commute, we have
\begin{equation}\begin{split}\label{oneterm}
 \tilde J_{n}-J_{n-1}&=  \frac{\left( (G_0 D^2 -G_0)\hat W^{\beta}_{n-1} ,  \hat W^{\beta}_{n-1}\right)}{ (W^{\beta}_{n-1})^2}\\
 &=  \frac{\left( (D^{2(1+n_0)} -D^2) \hat W^{\beta}_{n-1} ,  \hat W^{\beta}_{n-1}\right)}{ (W^{\beta}_{n-1})^2} \ge - I_n,
\end{split}\end{equation}
where the last inequality is obtained by neglecting the $D^{2(1+n_0)}$ term (which is non-negative) and using self-adjointness of $D$.
This contribution corresponds to the one generated by the $\Delta \hat Z^{\beta}_t(x)$-term in \eqref{dhatz}.
Now we want to estimate
$\bbE\left[ J_{n+1}-\tilde J_{n} \ | \ \cF_n\right]$.
The difficulty here is to deal with the $(W^{\beta}_{n})^{-2}$ terms whose conditional expectation cannot easily be computed. In order to obtain an upper bound we ``linearize'' this term using the inequality $(1+u)^{-2}\ge 1-2u$, which is valid for all $u\ge -1$. We obtain 
\begin{equation}\label{ineqq}
  \frac{1}{(W^{\beta}_{n})^2}=   \frac{1}{(W^{\beta}_{n-1})^2} \left(1+ \frac{W^{\beta}_{n}-W^{\beta}_{n-1}}{W^{\beta}_{n-1}} \right)^{-2} \ge  \frac{1}{(W^{\beta}_{n-1})^2} \left(1- 2 \frac{W^{\beta}_{n}-W^{\beta}_{n-1}}{W^{\beta}_{n-1}} \right).
\end{equation}
The analogue of this linearization in the continuous case would be to ignore the third term in the r.h.s.\ of \eqref{itoj} (a term which had no crucial role in the proof of Proposition~\ref{keyprop2dis}). 
Note that in the r.h.s.\ of \eqref{ineqq} we have the sum of an $\cF_{n-1}$-measurable term and a martingale increment.
Using \eqref{martindecompo}, we decompose both sides of the scalar product in a similar fashion and obtain
\begin{equation*}\begin{split}
  J_{n}&\ge  \left( 1- \frac{2(W^{\beta}_{n}-W^{\beta}_{n-1} ) }{W^{\beta}_{n-1}} \right)\frac{\left( G_0\hat W^{\beta}_{n},  \hat W^{\beta}_{n}  \right)}{(W^{\beta}_{n-1})^2}\\
  &=\left( 1- \frac{2(W^{\beta}_{n}-W^{\beta}_{n-1} ) }{W^{\beta}_{n-1}} \right)\frac{\left( G_0 D \hat W^{\beta}_{n-1} +G_0(\hat W^{\beta}_{n}- D \hat W^{\beta}_{n-1})  , D\hat W^{\beta}_{n-1} +(\hat W^{\beta}_{n}- D \hat W^{\beta}_{n-1})  \right)}{(W^{\beta}_{n-1})^2}.
\end{split}\end{equation*}
The r.h.s.\ in the above equation is a product of three terms (two of them vector-valued and one multiplication is the inner product). Expanding the product, we obtain 
\begin{multline}\label{8terms}
  J_{n}\ge \tilde J_n - \frac{2(W^{\beta}_{n}-W^{\beta}_{n-1} ) }{W^{\beta}_{n-1}} \tilde J_n
  + 2\frac{\big(G_0 D \hat W^{\beta}_{n-1} ,   \hat W^{\beta}_{n}- D \hat W^{\beta}_{n-1}  \big)}{(W^{\beta}_{n-1})^2}\\
  - \frac{4(W^{\beta}_{n}-W^{\beta}_{n-1} ) }{W^{\beta}_{n-1} }\frac{\big(G_0 D \hat W^{\beta}_{n-1} ,   \hat W^{\beta}_{n}- D \hat W^{\beta}_{n-1}  \big)}{(W^{\beta}_{n-1})^2}
  \\+ \frac{\big(G_0(\hat W^{\beta}_{n}- D \hat W^{\beta}_{n-1}) ,   \hat W^{\beta}_{n}- D \hat W^{\beta}_{n-1}  \big)}{(W^{\beta}_{n-1})^2}\\
  -\frac{2(W^{\beta}_{n}-W^{\beta}_{n-1} ) }{W^{\beta}_{n-1}}\frac{\big( G_0\big(\hat W^{\beta}_{n}- D \hat W^{\beta}_{n-1}\big) ,  \hat W^{\beta}_{n}- D \hat W^{\beta}_{n-1}  \big)}{(W^{\beta}_{n-1})^2}.
\end{multline}
There are eight terms in the expansion but this includes two identical pairs (corresponding to the third and fourth terms above)  due to the self-adjointness of $G_0$, and hence only six terms appear in the r.h.s.\ above.
The second and third terms are martingale increments, they correspond to the martingale part in the Itô formula, the  two following terms correspond (after taking conditional expectation) to the ``quadratic variation terms'' of the Itô formula (recall that compared to \eqref{itoj} one term has disappeared due to the linearization of the quotient \eqref{ineqq}). An additional term appears compared to \eqref{itoj} corresponding to the product of three martingale increments.

\medskip

The computations of the quadratic variation parts are similar to the continuous case and give similar expressions. We have
\begin{equation*}\begin{split}
\bbE\left[\frac{\big(G_0(\hat W^{\beta}_{n}- D \hat W^{\beta}_{n-1}) ,  \hat W^{\beta}_{n}- D \hat W^{\beta}_{n-1}  \big)}{(W^{\beta}_{n-1})^2} \ | \ \cF_{n-1} \right]
&= \chi(\beta) g_0(0) I_n\\
\bbE\left[
 \frac{W^{\beta}_{n}-W^{\beta}_{n-1}  }{W^{\beta}_{n-1}}\frac{ \big(G_0 D \hat W^{\beta}_{n-1} ,  \hat W^{\beta}_{n}- D \hat W^{\beta}_{n-1}  \big)}{(W^{\beta}_{n-1})^2}  \ | \ \cF_{n-1}\right]
 &= \chi(\beta)\Big(\big(D \mu^{\beta}_{n-1}\big)^2, G_0 D\mu^{\beta}_{n-1}\Big).
\end{split}\end{equation*}
The last term can also be computed in the same fashion,
\begin{multline}
\bbE\left[
\frac{(W^{\beta}_{n}-W^{\beta}_{n-1} ) }{W^{\beta}_n}\frac{ \big(G_0\big(\hat W^{\beta}_{n}- D \hat W^{\beta}_{n-1}\big) ,  \hat W^{\beta}_{n}- D \hat W^{\beta}_{n-1}  \big)}{(W^{\beta}_{n-1})^2} \ | \ \cF_{n-1} \right]
\\ = \chi_3(\beta)g_0(0)\sum_{x\in \bbZ^d } (D \mu^{\beta}_{n-1})(x)^3,
\end{multline}
where $\chi_3(\beta)=\bbE\left[ \left(e^{\go-\gl(\beta)}-1\right)^3\right].$
Hence taking the conditional expectation in \eqref{8terms} we obtain 
\begin{equation*}
  \bbE\left[J_{n}-\tilde J_n \ | \ \cF_{n-1}\right]\ge \chi(\beta) g_0(0) I_n-4\Big(\big(D \mu^{\beta}_{n-1}\big)^2, G_0 D\mu^{\beta}_{n-1}\Big)-\chi_3(\beta)g_0(0)\sum_{x\in \bbZ^d } (D \mu^{\beta}_{n-1})(x)^3
\end{equation*}
and we conclude by combining the above with \eqref{splitintwo} and \eqref{oneterm}.
\end{proof}

\begin{proof}[Proof of Lemma~\ref{Nquadratic}]
 Using  \eqref{borne} we have 
\begin{equation}\begin{split}
 \bbE\left[ (N_{n}-N_{n-1})^2 \ | \ \cF_{n-1}\right]&= \bbE\left[ (J_{n}-J_{n-1})^2 \ | \ \cF_{n-1}  \right]-\bbE\left[J_{n}-J_{n-1} \ | \ \cF_{n-1} \right]^2\\ & \le  \bbE\left[ J^2_{n}+J^2_{n-1} \ | \ \cF_{n-1}  \right] \\
 &\le n^2_0\left(\bbE\left[ \|\mu^{\beta}_{n}\|^4_2   \ | \ \cF_{n-1}  \right]+  \|\mu^{\beta}_{n-1}\|^4_2\right).
 \end{split}
 \end{equation}
 To bound the second term we observe that   $\|\mu^{\beta}_{n-1}\|^2_2\le 2d \|D\mu^{\beta}_{n-1}\|^2_2=2d I_n.$
To conclude, we need prove that  that
\begin{equation}\label{anotherineq}
 \bbE\left[ \|\mu^{\beta}_{\go,n+1}\|^4_2 \ | \ \cF_n  \right]\le e^{\frac{\gl(8\gb)+\gl(-8\gb)}{2}} I^2_n.
\end{equation}
The inequality \eqref{anotherineq} is proved (with a different constant) in \cite[Proof of Lemma 4.4]{J22}, we reproduce here the short proof  for the reader's  convenience. We introduce the  $\cF_{n-1}$-measurable probability
measure $\alpha_{n}(x)\coloneqq  \frac{ (D\mu^{\beta}_{n-1})(x)^2 }{I_n}$. We have
\begin{equation}
\|\mu^{\beta}_{n}\|^2_2= \frac{\sum_{x\in \bbZ^d} \alpha_{n}(x) e^{2\beta \go_{n,x}}}{\big(\sum_{x\in \bbZ^d} (D\mu^{\beta}_{n-1}(x)) e^{\beta \go_{n,x}}\big)^2}  I_n,
\end{equation}

and hence
\begin{multline}
 \frac{\bbE\left[ \|\mu^{\beta}_{n}\|^4_2 \ \middle | \ \cF_{n-1}\right]}{I_n^2}\le  \bbE\left[ \frac{\left(\sum_{x\in \bbZ^d} \alpha_{n}(x) e^{2\beta \go_{n,x}}\right)^2}{\left(\sum_{x\in \bbZ^d} (D\mu^{\beta}_{n-1}(x)) e^{\beta \go_{n,x}}\right)^4}  \middle |\ \cF_{n-1} \right]\\ \le \sqrt{ \bbE\bigg[ \bigg(\sum_{x\in \bbZ^d} \alpha_{n}(x) e^{2\beta \go_{n,x}}\bigg)^4  \Big |\ \cF_{n-1} \bigg]  \bbE\bigg[ \bigg(\sum_{x\in \bbZ^d} (D\mu^{\beta}_{n-1}(x)) e^{\beta \go_{n}}\bigg)^{-8}  \Big  |\ \cF_{n-1} \bigg]}
   \le e^{\frac{\gl(8\gb)+\gl(-8\gb)}{2}},
\end{multline}
 the last inequality being Jensen's inequality for the probability measures $\alpha_{n}$ and $D\mu^{\beta}_{n-1}$.
\end{proof}

\appendix

\section{Proof of Theorem~\ref{b2bcrit}}\label{gapfilling}

Let us provide some details on how the strict inequality $\beta_c>\beta_2$ can be extracted from \cite{BT10}.
We want to show that there exists a $\beta>\beta_2$ such that $W^{\beta}_n$ is uniformly integrable.
The starting point is to observe that $W^{\beta}_n$ is uniformly integrable if for some $\gamma\in(0,1)$ we have
\begin{equation}\label{bonmoment}
\sup_{n\ge0}\bbE[ (W^{\beta}_n)^{1+\gamma}]=\sup_{n\ge0}\tilde\bbE_n[ (W^{\beta}_n)^{\gamma}]<\infty.
\end{equation}
Using Lemma~\ref{sbrepresent} and Jensen's inequality we have
\begin{equation}\label{jensenee}
\tilde\bbE_n[ (W^{\beta}_n)^{\gamma}]= E\otimes \bbE\otimes \hat \bbE\left[\left(\tilde W^{\beta}_n\right)^{\gamma} \right]
\le  E \left[\left(\bbE\otimes \hat \bbE \left[\tilde W^{\beta}_n\right]\right)^{\gamma} \right]
\end{equation}
where $\tilde W^{\beta}_n\coloneqq  E'\left[e^{\beta H_n(\tilde \go,X')-n\gl(\beta)} \right]$  ($X'$ is a simple random walk on $\bbZ^d$ with distribution denoted by $P'$) is the partition function corresponding to environment $\tilde \go$ .
We have 
\begin{equation}
\bbE\otimes \hat \bbE \left[\tilde W^{\beta}_n\right]= E'\otimes\bbE\otimes \hat \bbE \left[e^{\beta H_n(\tilde \go,X')-n\gl(\beta)}\right]   =E'\left[ e^{\left(\gl(2\beta)-\gl(\beta)\right)\sum^n_{i=1}\ind_{\{X_i=X'_i\}}} \right].
\end{equation}
The right-hand side corresponds to the partition function of the Random Walk Pinning  model (RWPM) studied in \cite{BT10,BS10,BS11}. It is a function of the random walk $X$ and is defined as 
\begin{equation}
 \bar Z^h_{n,X}\coloneqq  E'\left[ e^{h\sum^n_{i=1}\ind_{\{X_i=X'_i\}}} \right].
\end{equation}
Thus to show that \eqref{bonmoment} holds for some $\beta>\beta_2$,  it is sufficient to show that there exists an
$$h>h_0\coloneqq -\log \left(P^{\otimes 2}\left[\exists n\ge 1, X^{(1)}_n=X^{(2)}_n \right]\right)$$ such that  
$\sup_n\bbE\left[(\bar Z^h_{n,X})^{\gamma}\right]<\infty$ (recall that from \eqref{defbeta2} we have $\gl(2\beta_2)-2\gl(\beta_2)= h_0$).

\medskip
Letting $T_n(X,X')\coloneqq \sup\{ m\le n \ \colon  \ X_m=X'_m  \}$ we have
\begin{multline*}
  \bar Z^h_{n,X}
= \! \sum_{m=0}^n E'\! \left[ e^{h\sum^n_{i=0}\ind_{\{X_i=X'_i\}}}\ind_{\{T_n=m\}} \right]\le  \sum_{m=0}^n E' \! \left[ e^{h\sum^m_{i=0}\ind_{\{X_i=X'_i\}}}\ind_{\{X_m=X'_m\}} \right]=\colon \!\! \sum^n_{m=0}    Z^h_{m,X},
\end{multline*}
where $Z^h_{m,X}$ is the \textit{constrained} partition function of the RWPM.
As $\gamma\in(0,1)$ we have thus 
\begin{equation}
 \sup_{n\ge 0}\bbE\left[  (\bar Z^h_{n,X})^{\gamma}\right] \le \sum_{m\ge 1} \bbE[ (Z^h_{m,X})^{\gamma}].
\end{equation}
Thus to conclude, it is sufficient to prove that $u^h_n\coloneqq \bbE[ (Z^h_{m,X})^{\gamma}]$ is summable in $m$ for some value of $\gamma\in(0,1)$. In \cite{BT10} it is only shown that for some $h>h_0$ and  $\gamma=6/7$,  $u_{mL}$ is \textit{uniformly bounded} in $m$ for some fixed large integer $L$ (see \cite[Equation (41)]{BT10}). However the desired result can be rather easily obtained using some earlier estimates present in that proof. Let us provide the details here.
We show that there exists  an $h>h_0$, an integer $L\ge 1$ and $C\ge 0$ such that 
\begin{equation}\begin{split}\label{twothings}
\forall i\in \lint 0,L-1\rint,\ \forall m\ge 1,&\ \quad  u^h_{mL-i}\le C u^h_{mL},\\
 \sum_{n\ge 0} u^h_{mL}&<\infty.
 \end{split}
\end{equation}
The first line is valid for any choice of $h$ and $L$ provided that $C$ is chosen sufficiently large. It is an immediate consequence of the supermultiplicativity  $u_{n+k}\ge u_n u_k$ which itself follows from 
\begin{equation}
 Z^h_{n,X}\ge  E'\left[ e^{h\sum^{n+1}_{i=0}\ind_{\{X_i=X'_i\}}}\ind_{\{X_n=X'_n \text{ and } X_{n+k}=X'_{n+k}\}}\right]= Z^h_{n,X}Z^h_{k,X^{(n)}},
\end{equation}
where the random walk $X^{(n)}$ is defined by $X^{(n)}_i=X_{n+i}-X_n$. We  now show 
how the second line in \eqref{twothings} can be 
extracted from the proof presented in \cite[Section 3]{BT10}. 
Following that reference setting $\gamma=\frac{6}{7}$, then given $\delta>0$, we can find $h>h_0$, $L\ge 1$ and  $C>0$ such that for every $m\ge 1$, we have
\begin{equation}\label{ziop}
 u^h_{mL}\le  C\sum_{I\subset \lint 1,m\rint} \ind_{\{m\in I\}}\prod_{j=1}^{|I|} \frac{\delta}{(i_{j}-i_{j-1})^{6/5}},
\end{equation}
where $1\le i_1<\dots<i_{|I|}$ denotes by convention the ordered elements of the set $I\subset \lint 1,m\rint$ and by convention $i_0=0$.
More precisely \eqref{ziop} appears as \cite[Equation (38)]{BT10} but without the indicator $\ind_{\{m\in I\}}$ \footnote{There is also another harmless difference: The partition function considered in \cite[Equation (38)]{BT10} is not $Z^h_{mL,X}$ but a variant of it which differs by a multiplicative constant see \cite[Definition 2.9(5)]{BT10})}. However the reader can without difficulty see that the proof presented in \cite{BT10} yields the stronger inequality  \eqref{ziop} without the need for any modification (cf. the observation made below \cite[Equation (18)]{BT10}). 

\medskip

To conclude, we observe that \eqref{ziop} implies that $u^h_{mL}$ is summable. This is a consequence of the lemma stated below, which is a particular case of a classical result for terminating renewal processes. 
\begin{lemma}
 If $\varrho\coloneqq  \delta \sum_{n\ge 1} n^{-6/5}<1$,  then we have the following asymptotic equivalence when $m\to \infty$ 

 \begin{equation}
  \sum_{I\subset \lint 1,m\rint} \ind_{\{m\in I\}}\prod_{j=1}^{|I|} \frac{\delta}{(i_{j}-i_{j-1})^{6/5}} =  \frac{(1+o(1))\delta m^{-6/5}}{(1- \varrho)^2}.
 \end{equation}

\end{lemma}
We refer to \cite[Theorem 2.2]{GiacBook} for a more general statement and its proof. More precisely, the lemma is obtained by applying \cite[Theorem 2.2(2)]{GiacBook} to the specific case $K(n)= \delta n^{-6/5}$ (the parameter $\beta$ mentioned in \cite{GiacBook} being equal to zero).
\qed

\section{Proof of Corollary~\ref{notessential}} \label{ouix}
We can without loss of generality assume that $\log u\ge 1$, since the result is easy to prove otherwise.
Note that we have 
\begin{align}
 \left\{\max_{(n,x)\in [0,C\log u]\times \bbZ^d}  \hat W_n(x)\ge u\right\}
 &\supseteq  \left\{\max_{(n,x)\in \bbN\times \bbZ^d}  \hat W_n(x)\ge u\right\} \setminus
 \left\{  \max_{ n> C \log u} W^{\beta}_n \ge 1\right\},\\
 \left\{\max_{n\in [0,C\log u]}   W_n\ge u\right\}
 &\supseteq  \left\{\sup_{n\in \bbN}   W_n\ge u\right\} \setminus
\left\{  \max_{ n> C \log u} W^{\beta}_n \ge 1\right\}.
\end{align}
In view of Theorem~\ref{locaend}  and Lemma~\ref{dull}, it is sufficient to show that the event
$\max_{ n> C \log u} W^{\beta}_n \ge 1$ has small probability, which is achieved using the following inequality (valid for some positive constant $\alpha$)
\begin{equation}\label{highrize}
 \bbP\left( W^{\beta}_n \ge 1\right) \le e^{-\alpha n}.
\end{equation}
To see that \eqref{highrize} holds, recall that from  \eqref{freen}  we have
 $\bbE\left[ \log W^{\beta}_n \right] \le n\f(\beta)$  and that $\f(\beta)<0$ by assumption.
 Hence 
 \begin{equation}
  \bbP\left( W^{\beta}_n \ge 1\right)\le   \bbP\left(\left( \log W^{\beta}_n- \bbE\left[ \log W^{\beta}_n\right]\right) \ge  -n \f(\beta)\right).
 \end{equation}
We the conclude by using an exponential concentration inequality proved as \cite[Theorem 6.1]{LW09}, which states that, under the assumption \eqref{expomoment}, for any $u\ge 0$ there exists  a constant $\theta(u)>0$ such that
\begin{equation}
 \bbP\left( \log W^{\beta}_n- \bbE\left[ \log W^{\beta}_n\right]\ge  un\right)\le e^{- \theta(u) n},
\end{equation}
and obtain \eqref{highrize} with $\alpha=\theta(-\f(\beta))$.
 \qed

\section*{Acknowledgements}
The authors wish to thank  Q.\ Berger, P.\ Carmona, R.\ Fukushima, Y.\ Hu and A.\ Teixeira for helpful comments. We are also grateful to two anonymous referees for their careful reading and helpful feedback.
This work was partially realized as H.L.\ was visiting  the Research Institute for Mathematical Sciences, an International Joint Usage/Research Center located in Kyoto University and he  acknowledges kind hospitality and support. 
This research was supported by JSPS Grant-in-Aid for Scientific Research 23K12984 and 21K03286. H.L.\ also acknowledges the support of a productivity grand from CNQq and of a CNE grant from FAPERj.

\bibliographystyle{plain}
\bibliography{ref.bib}

\end{document}